\documentclass{article}
\usepackage[utf8]{inputenc}
\usepackage[english]{babel}
\usepackage[a4paper]{geometry}
\usepackage[]{amsthm}
\usepackage{amsmath}
\usepackage{footnote}
\usepackage[]{amssymb}
\usepackage{xfrac}
\usepackage{enumitem}
\usepackage{titling}
\usepackage{graphicx}
\usepackage{MnSymbol}
\usepackage{tikz}
\usepackage{tikz-cd}
\usetikzlibrary{decorations.pathreplacing,angles,quotes}
\usetikzlibrary{babel}
\theoremstyle{definition}
\newtheorem{theorem}{Theorem}[section]

\newtheorem{prop}[theorem]{Proposition}
\newtheorem{lemma}[theorem]{Lemma}
\newtheorem{definition}{Definition}[section]
\renewcommand{\hom}{Hom}
\newcommand{\Z}{\mathcal{Z}}
\newcommand{\id}{\text{id}}
\newcommand{\Aut}[1]{\text{Aut(#1)}}
\DeclareMathOperator{\im}{\text{im}}
\title{The Cohomology Class of the Mod 4 Braid Group}
\author{Trevor Nakamura}
\date{}
\begin{document}
\maketitle
\begin{abstract}
The mod 4 braid group, $\Z_{n}$, is defined to be the quotient of the braid group by the subgroup of the pure braid group generated by squares of all elements, $PB_{n}^{2}$. Kordek and Margalit proved $\Z_{n}$ is an extension of the symmetric group by $\mathbb{Z}_{2}^{\binom{n}{2}}$. For $n\geq 1$, we construct a 2-cocycle in the group cohomology of the symmetric group with twisted coefficients classifying $\Z_{n}$. We show this cocycle is the$\mod2$ reduction of the 2-cocycle corresponding to the extension of the symmetric group by the abelianization of the pure braid group. We also construct the 2-cocycle corresponding to this second extension and show it represents an order two element in the cohomology of the symmetric group. Furthermore, we give presentations for both extensions and a normal generating set for $PB_{n}^{2}$. 
\end{abstract}
\section{Introduction}
In 2014 Brendle and Margalit proved the level 4 congruence subgroup, $B_{n}[4]$, of the braid group, $B_{n}$, is the subgroup of the pure braid group generated by squares of all elements, $PB_{n}^{2}$ \cite{B/M}. More recently, Kordek and Margalit showed the quotient of the braid group by $PB_{n}^{2}$, which we define as the mod 4 braid group and denote $\Z_{n}$, is an extension of the symmetric group, $S_{n}$, by $\mathbb{Z}_{2}^{n\choose2}$ \cite{K/M}. Eilenberg and MacLane proved we can classify group extensions by low dimensional cohomology classes in the group cohomology with twisted coefficients \cite{E/M}. In particular, $\Z_{n}$ is classified by a class in the second cohomology of the symmetric group with coefficients in $\mathbb{Z}_{2}^{n\choose2}$ twisted by the action of the symmetric group permuting unordered pairs of integers. We begin by constructing a CW-complex which is the two-skeleton of an Eilenberg-MacLane space for the symmetric group. This Eilenberg-MacLane space gives us a low dimensional approximation for a resolution of $\mathbb{Z}$ over $\mathbb{Z}S_{n}$. Then we construct a presentation for the extension of the symmetric group by the abelianization of the pure braid group. We obtain our main results from classifying the extension of the symmetric group by the abelianization of the pure braid group and composing this 2-cocycle with the$\mod$2 reduction of integers.

Let $G$ be a group and $A$ be a $G$-module. We represent the, possibly nontrivial, action of $G$ on $A$ by a map $\theta:G\to\Aut{A}$. For an extension $E$ of $G$ by $A$, with $\iota:A\to E$ and $\pi:E\to G$, we say $E$ gives rise to $\theta$ if conjugating an element of $\iota(A)$ by any $e\in E$ is determined by $\theta(\pi(e))$. The cohomology of $G$ with coefficients in $A$ is equivalent to the cohomology of an Eilenberg-MacLane space for $G$, referred to as a $\mathcal{K}(G,1)$-space, with a local coefficients system of $A$ determined by $\theta$. We use $H^{2}(G;A)$ to denote the second cohomology group of $G$ with coefficients in $A$ twisted by $\theta$. Fixing the action of $G$ on $A$, Eilenberg and MacLane proved there exists a bijection between $H^{2}(G,A)$ and the set of group extensions $E$, up to equivalence, of $G$ by $A$ which give rise to $\theta$. In this paper we will only consider the standard action of the symmetric group on unordered pairs of integers between 1 and $n$; therefore we often omit $\theta$. We give a more explicit explanation of the background for group cohomology in section 2.2.

Extensions of the symmetric group arise naturally while studying quotients of the braid group, $B_{n}$. The level $m$ congruence subgroup of the braid group, $B_{n}[m]$, is defined more generally as the kernel of a map: $B_{n}\to GL_{n}(\mathbb{Z}_{m})$ (more information is given in section 2.3). In 2018 Kordek and Margalit showed $\mathcal{Z}_{n}=B_{n}/B_{n}[4]$ is an extension of $S_{n}$ by $\mathbb{Z}_{2}^{n\choose 2}$ where $S_{n}$ permutes the generators of $\mathbb{Z}_{2}^{n\choose 2}$ by the standard action on unordered pairs of $\{1,\ldots, n\}$. Further work by Appel, Bloomquist, Gravel, and Holden prove that for $m$ odd, $B_{n}[m]/B_{n}[4m]\approx\Z_{n}$ \cite{A/B/G/H}. In this paper we will give descriptions of $\Z_{n}$ by both the corresponding 2-cocycle in $H^{2}(S_{n};\mathbb{Z}_{2}^{n\choose2})$ and by a group presentation. Our first theorem shows the cocycle classifying $\Z_{n}$ is determined by a cocycle which classifies an extension of $S_{n}$ by $\mathbb{Z}^{n\choose 2}$.
\begin{theorem}\label{main theorem}
	If $n\geq 1$, then the cohomology class $[\kappa]\in H_{}^{2}(S_{n};\mathbb{Z}_{2}^{n\choose2})$ classifying $\Z_{n}$ as an extension of $S_{n}$ by $\mathbb{Z}_{2}^{n\choose2}$ is the$\mod 2$ reduction of an element $[\phi]\in H^{2}_{}(S_{n};\mathbb{Z}^{n\choose2})$ of order $2$.
\end{theorem}
We define representatives of our cohomology classes by constructing a low dimensional approximation for a cellular chain complex of the universal cover of a $\mathcal{K}(S_{n},1)$-space. Furthermore, we show that the 2-chains of this chain complex are generated by the $S_{n}$ orbits of three classes of elements: $\tilde{c}_{i,j}$, $\tilde{d}_{i,j,k,\ell}$, and $\tilde{e}_{i,k,j}$ where $1\leq i<j\leq n$ and $1\leq k<\ell\leq n$. Respectively, each of these three types of generators correspond to the squaring, commuting, and braid relations in the presentation for the symmetric group. Note that the generators $\tilde{e}_{i,k,j}$ are only required for $n\geq 3$ while  $\tilde{d}_{i,j,k,\ell}$ are required only for $n\geq 4$. Therefore, the cocycle representing $[\kappa]$ is determined as a function from the $S_{n}$ module generated by these three classes of 2-chains to $\mathbb{Z}_{2}^{n\choose2}$. Using the presentation:
\begin{equation}\label{pres of Z mod 2}
	\mathbb{Z}_{2}^{n\choose2}=\langle \{\bar{g}_{i,j}\}_{1\leq i\neq j\leq n} \mid \bar{g}_{i,j}^{2}=1,\hspace{.5em}[\bar{g}_{i,j},\bar{g}_{k,\ell}]=1\rangle
\end{equation}
our second theorem defines the 2-cocycle classifying $\Z_{n}$.
\begin{theorem}\label{def of cocycle}
	Let $i<j$ and $k<\ell$. If $n\geq 1$, then a representative for the cocycle classifying $\Z_{n}$ as an extension of $S_{n}$ by $\mathbb{Z}_{2}^{n\choose2}$ is given by:
	\begin{align*}
		\kappa(\tilde{c}_{i,j})&=\bar{g}_{i,j}\\
		\kappa(\tilde{d}_{i,j,k,\ell})&=
		\begin{cases}
			\bar{g}_{i,k}+\bar{g}_{k,j}+\bar{g}_{i,\ell}+\bar{g}_{\ell, j}\hspace{1em}&i<k<j<\ell\text{ or }k<i<\ell<j\\
			0\hspace{1em}&\text{otherwise}
		\end{cases}\\
		\kappa(\tilde{e}_{i,k,j})&=
		\begin{cases}
			\bar{g}_{i,j}+\bar{g}_{j,k}\hspace{1em}&i<k<j,\hspace{.5em}j<i<k,\text{ or }k<j<i\\
			0\hspace{1em}&\text{otherwise}
		\end{cases}
	\end{align*}
\end{theorem}
 In the construction of Theorem \ref{def of cocycle} we give a presentation for $\mathcal{Z}_{n}$ with relations in Table \ref{Relations of Z}. Furthermore, considering Artin's original presentation for the braid group given by:
\begin{equation}\label{Artin braid}
	B_{n}=\left\langle b_{1},\ldots,b_{n}\middle\mid\hspace{.5em}\begin{aligned}& b_{i}b_{i+1}b_{i}=b_{i+1}b_{i}b_{i+1}\text{ for all }i\\
	&b_{i}b_{j}=b_{j}b_{i}\text{ if }|i-j|>1
	\end{aligned}\right\rangle
\end{equation}
our presentation for $\mathcal{Z}_{n}$ yields a normal generating set for $PB_{n}^{2}$ as a subgroup of $B_{n}$. Since Brendle and Margalit proved $PB_{n}^{2}=B_{n}[4]$, we get a normal generating set for $B_{n}[4]$ in Theorem \ref{normal gen set for B_{n}}. To simplify notation in the statement of Theorem \ref{normal gen set for B_{n}}, let $b_{i,j}$ be the conjugation of $b_{j-1}$ by $b_{i}b_{i+1}\cdots b_{j-2}$. 
\begin{theorem}\label{normal gen set for B_{n}}
	For all $n\geq 1$, $B_{n}[4]$ is normally generated as a subgroup of the braid group by elements of the form:
	\begin{enumerate}
		\item $[b_{i}^{2},b_{i+1}^{2}]$ for all $1\leq i\leq n-1$.
		\item $[b_{i,i+2}^{2},b_{i+1,i+3}^{2}]$ for all $1\leq i\leq n-3$
		\item $b_{i}^{4}$ for all $1\leq i\leq n-1$
	\end{enumerate}
\end{theorem}
\paragraph{Outline}
This paper will begin with preliminaries which will help with constructing group presentations, group cohomology, and braid/symmetric groups. In section 3 we build a truncated resolution for the chain complex corresponding to the universal cover of a $\mathcal{K}(S_{n},1)$ space and define the maps needed to construct a 2-cocycle using this resolution. Then in section 4 we construct a representative for the cohomology class $[\phi]\in H^{2}(S_{n};\mathbb{Z}^{n\choose 2})$ and compute the order of this element. In section 5 we prove $\kappa$ is determined by $\phi$, finishing the proof of Theorems \ref{main theorem} and \ref{def of cocycle}. The proof of Theorem \ref{normal gen set for B_{n}} is finished at the end of the paper with an explanation of the difficulties which arise in finding a finite generating set for $B_{n}[4]$.
\subsection{Acknowledgments}
First I would like to give my deepest gratitude to my advisor Matthew Day for his continuous guidance, suggestions, and comments throughout this research and early drafts of this paper. I would also like to thank Dan Margalit and Matthew Clay for helpful conversations. Many thanks also to the anonymous referee for their suggestions and comments.
\section{Preliminaries}
\subsection{Building a Group Presentation}\label{section on building presentations}
Given a group extension
$$\begin{tikzcd}
	1\arrow{r}&K\arrow{r}{\iota}&G\arrow{r}{\pi}&Q\arrow{r}&1
\end{tikzcd}$$
and group presentations for $K$ and $Q$, we describe the process to build a group presentation for $G$. This construction is previously known and included for completeness. Let $\langle S_{K}\mid R_{K}\rangle$ and $\langle \bar{S}_{Q}\mid \bar{R}_{Q}\rangle$ be group presentations for $K$ and $Q$ respectively. Note that, without loss of generality, we may assume $1_{Q}\in \bar{S}_{Q}$. For any set $S$, let $F(S)$ denote the free group generated by $S$. Since we consider $K$ as a subgroup of $G$ we will not distinguish between $K$ and $\iota(K)$.
\paragraph{Generators:}
Since $Q=G/K$, we may lift any element $\bar{g}\in Q$ to an element $g\in G$ by choosing a $g$ such that $\pi(g)=\bar{g}$. Fix $S_{Q}\subset G$ as choice of lifts from elements in $\bar{S}_{Q}$ such that $\pi$ restricted to $S_{Q}$ is a bijection. Note that if $1_{Q}\in\bar{S}_{Q}$, we may choose the lift of $1_{Q}$ to be $1_{G}$. Therefore we may assume $1_{Q}\notin\bar{S}_{Q}$. Define $\mathcal{S}=S_{K}\cup S_{Q}$ as the generators of $G$.
\paragraph{Relations:}
Since $K\approx \langle S_{K}\mid R_{K}\rangle$, any element $k\in K$ can be represented by a word $k_{1}\cdots k_{q}\in F(S_{K})$ for some $q\geq 1$ and $k_{i}\in S_{K}$ for all $i$. Furthermore, since $K$ is normal in $G$, both $sks^{-1}$ and $s^{-1}ks$ are in $K$ for any $s\in S_{Q}$ and $k\in K$; the choice of $s\in S_{Q}$ determines element $k',k''\in K$ such that $sks^{-1}=k'$ and $s^{-1}ks=k''$. For each $s\in S_{Q}$ and $k\in S_{K}$, fix a choice of words $x_{k}$ and $y_{k}$ in $F(S_{K})$ representing $k'$ and $k''$ respectively. For all $s\in S_{Q}$ and $k\in S_{K}$, let $R'$ be the relations in $G$ defined by $sks^{-1}=x_{k}$ and $s^{-1}ks=y_{k}$. 
	
Now, any relation in $\bar{R}_{Q}$ is given by $\bar{s}_{1}\cdots\bar{s}_{q}=1_{Q}$ where $q\geq 1$ and $\bar{s}_{i}\in \bar{S}_{Q}$ for all $1\leq i\leq q$. Since, for each $i$, $s_{i}\in S_{Q}$ is defined to be a lift of $\bar{s}_{i}$ and $G$ is an extension of $Q$ by $K$, there exists some $k\in K$ such that $s_{1}\cdots s_{q}=k$. Furthermore, the choices of $s_{i}$ as a lift of $\bar{s}_{i}$ uniquely determines $k\in K$. Therefore, $k$ is expressed as a word $k_{1}\cdots k_{t}$ on the generators of $S_{K}$, for some $t\geq 1$. Hence $s_{1}\cdots s_{q}=k_{1}\cdots k_{t}$ determines the relation $s_{1}\cdots s_{q}k_{t}^{-1}\cdots k_{1}^{-1}=1_{G}$ in $G$. Let $R_{Q}$ be the set of all relations in $G$ determined by lifting relations of $\bar{R}_{Q}$ using this method. Define $R=R_{K}\cup R_{Q}\cup R'$ and $G'=\langle \mathcal{S}\mid R\rangle$.
	
\paragraph{Isomorphism} Given the definition for $G'$ above, it remains to show $G'\approx G$. Consider the natural map $f:G'\to G$ defined by $f(s)=s$ for all $s\in \mathcal{S}$, it suffices to show there exists homomorphisms $\iota'$ and $\pi'$ such that following diagram has exact rows and commutes:
\begin{equation}\label{commuting diagram for presentation}
	\begin{tikzcd}
	1\arrow{r}&K\arrow{r}{\iota'}\arrow{d}{\id_{K}}&G'\arrow{r}{\pi'}\arrow{d}{f}&Q\arrow{r}\arrow{d}{\id_{Q}}&1\\
	1\arrow{r}&K\arrow{r}{\iota}&G\arrow{r}{\pi}&Q\arrow{r}&1
	\end{tikzcd}
\end{equation}
\begin{lemma}\label{lemma: iota' is injective}
	Let $G$ be the extension of $Q$ by $K$ above and $G'=\langle \mathcal{S}\mid R\rangle$ be the group described in the preceding paragraphs. There exists $\iota':K\to G$ such that $\iota'$ is injective and commutes with \eqref{commuting diagram for presentation}. 
\end{lemma}
\begin{proof}
	Recall $\langle S_{K}\mid R_{K}\rangle$, $\langle \bar{S}_{Q}\mid \bar{R}_{Q}\rangle$ and $\langle \mathcal{S}\mid R\rangle$ are group presentations for $K,Q$ and $G'$ respectively. Since $K\approx\langle S_{K}\mid R_{K}\rangle$, any element $k\in K$ can be written as $k_{1}\cdots k_{t}$ where $k_{i}\in S_{K}$ for all $i$. Define $\iota':K\to G'$ by $\iota'(k)=k_{1}\cdots k_{t}$. Since $R_{K}\subset R$, $\iota'$ is a well defined homomorphism. Furthermore, since $f:G'\to G$ defined above is natural, $f\circ\iota'(k)=\iota\circ\id_{K}(k)$ for all $k\in K$. 
\end{proof}
\begin{lemma}\label{lemma: rewritting}
	Let $K,G',Q$ be groups with presentations be defined as above. Any element $g\in G'$ can be written as $g=w_{1}w_{2}$ where $w_{1}$ and $w_{2}$ are words in the generators of $Q$ and $K$ respectively.
\end{lemma}
\begin{proof}
	Let $g\in G'$, then $g=s_{1}\cdots s_{t}$ where $s_{i}\in S$. By induction on $t$, suppose $g=s_{1}s_{2}$. In this case the statement is trivial unless $s_{1}\in S_{k}$ and $s_{2}\in S_{Q}$. By the relations of $R'$ (conjugating $S_{K}$ by elements of $S_{Q}$ and the inverses) there exists $k\in K$ such that $s_{2}^{-1}s_{1}s_{2}=k$. Consider
	\begin{align*}
		s_{1}s_{2}&=s_{2}s_{2}^{-1}s_{1}s_{2}\\
		&=s_{2}k
	\end{align*}
	Since $K$ is generated by $S_{K}$, $k$ is represented by a word $w_{2}\in F(S_{K})$ and hence the Lemma holds for $t=2$.
	
	Now, suppose the Lemma is true for any element of $G'$ that can be represented by a reduced word of length $t-1$ in the generators of $\mathcal{S}$. Suppose $g=s_{1}\cdots s_{t}$, then by induction $s_{1}\cdots s_{t-1}=w_{1}w_{2}$ where $w_{1}$ and $w_{2}$ are words in the generators of $S_{Q}$ and $S_{K}$ respectively. If $s_{t}\in S_{K}$, then we are done. Therefore assume $s_{t}\in S_{Q}$, then we have $g=w_{1}w_{2}s_{t}$. In particular, $w_{2}=k_{1}\cdots k_{\ell}$ for $k_{1},\ldots, k_{\ell}\in S_{K}$. By the relations of $R'$ there exists $h_{i}\in K$ such that $s_{t}^{-1}k_{i}s_{t}=h_{i}$ for every $i$ such that $1\leq i\leq \ell$. Thus we have:
	\begin{align*}
		s_{1}\cdots s_{t}&=w_{1}w_{2}s_{t}\\
		&=w_{1}k_{1}\cdots k_{\ell}s_{t}\\
		&=w_{1}k_{1}\cdots k_{\ell-1}s_{t}s_{t}^{-1}k_{\ell}s_{t}\\
		&=w_{1}k_{1}\cdots k_{\ell-2}s_{t}s_{t}^{-1}k_{\ell-1}s_{t}h_{\ell}\\
		&\hspace{.4em}\vdots\\
		&=w_{1}s_{t}h_{1}\cdots h_{\ell}
	\end{align*}
	Noting that each $h_{i}$ can be written as a word in $F(S_{K})$ for all $i$ proves the Lemma.
\end{proof}

\begin{lemma}\label{lemma: def of pi' and surjective}
	Let $K,G',Q$ with group presentations defined previously. There exists $\pi':G'\to Q$ such that \eqref{commuting diagram for presentation} commutes.
\end{lemma}
\begin{proof}
	Let $K,G,Q$, and $G'$ be as above. Recall $G'=\langle \mathcal{S}\mid R\rangle$ where $\mathcal{S}=S_{Q}\cup S_{K}$ and $R=R'\cup R_{Q}\cup R_{K}$. Without loss of generality, assume $1_{Q}\notin \bar{S}_{Q}$. Define $\pi':G\to Q$ by the following:
	$$\pi'(s)=\begin{cases}
		\bar{s}\hspace{1em}&s\in S_{Q}\\
		1\hspace{1em}&s\in S_{K}
	\end{cases}$$
	We first show $\pi'$ is well defined by proving $\pi'$ respects the relations in $R$. If $r$ is a relation in $R_{K}$, then $r$ can be expressed as a word in $F(S_{K})$ and $\pi'(r)=1$. By construction any relation $r\in R_{Q}$ is $s_{1}s_{2}\cdots s_{q}k_{t}^{-1}k_{t-1}^{-1}\cdots k_{1}^{-1}=1$ where $s_{i}\in S_{Q}$ for all $1\leq i\leq q$ and $k_{j}\in S_{Q}$ for all $1\leq j\leq t$. Furthermore:
	\begin{align*}
		\pi'(s_{1}s_{2}\cdots s_{q}k_{t}^{-1}k_{t-1}^{-1}\cdots k_{1}^{-1})&=\pi'(s_{1})\pi'(s_{2})\cdots\pi'(s_{q})\pi'(k_{t})^{-1}\pi'(k_{t-1})^{-1}\cdots \pi'(k_{1})^{-1}\\
		&=\bar{s}_{1}\bar{s}_{2}\cdots\bar{s}_{q}
	\end{align*}
	By the choice of $S_{Q}$ as a lift of $\bar{S}_{Q}$, $\bar{s}_{1}\bar{s}_{2}\cdots \bar{s}_{q}=1$ is a relation in $Q$. 
	
	Now, suppose $r$ is a relation in $R'$. Then $r$ is of the form $sks^{-1}=k'$ for some $s\in S_{Q}$, $k\in S_{K}$, and $k'\in K$. Since $k'\in K$, $k'$ can be represented by a word $k_{1}k_{2}\cdots k_{t}$ in $F(S_{K})$, $\pi'(k')=1$. It remains to show $\pi'(sks^{-1})=1$:
	\begin{align*}
		\pi'(sks^{-1})&=\pi'(s)\pi'(k)\pi'(s^{-1})\\
		&=\bar{s}\cdot 1\cdot \bar{s}^{-1}\\
		&=1
	\end{align*}
	Thus $\pi'$ is a well defined homomorphism. Furthermore, by the definition of $\mathcal{S}$, $\pi'$ is a surjection. 
	
	It remains to show $\pi'$ commutes with \eqref{commuting diagram for presentation}. Let $g\in G'$, then $g$ can be represented by $s_{1}s_{2}\cdots s_{q}$ for some word in $F(\mathcal{S})$. Since we assume $1_{Q}\notin \bar{S}_{Q}$, notice that $\pi\circ f(g)=\bar{s}_{1}\bar{s}_{2}\cdots\bar{s}_{q}$ where $\bar{s}_{i}=1_{Q}$ if and only if $s_{i}\in S_{K}$. This is the definition of $\pi'$, so $\pi'$ commutes with \eqref{commuting diagram for presentation}.
\end{proof}
\begin{theorem}\label{thm: presentation is exact}
	Let $K,G'$ and $Q$ be as above with $\iota'$ and $\pi'$ as in the previous lemmas. Then the top row of \eqref{commuting diagram for presentation} is exact. 
\end{theorem}
\begin{proof}
	By Lemma \ref{lemma: iota' is injective} and Lemma \ref{lemma: def of pi' and surjective}, it suffices to show $\im\iota'=\ker\pi'$. By the definition of $\pi'$, $\im\iota'\subseteq\ker\pi'$. It remains to show $\ker\pi'\subseteq\im\iota'$. Suppose $g\in\ker\pi'$, by lemma \ref{lemma: rewritting}: $g=w_{1}w_{2}$ for some words $w_{1}$ and $w_{2}$ on the generators of $S_{Q}$ and $S_{K}$ respectively. Furthermore, $w_{1}=s_{1}\cdots s_{q}$ where $s_{i}\in S_{Q}$ for all $i$. Since $w_{2}$ can be expressed as a word in $F(S_{K})$, $\pi'(g)=\bar{s}_{1}\bar{s}_{2}\cdots\bar{s}_{q}$ where $\bar{s}_{i}\in \bar{S}_{Q}$ for all $i$. But $g\in\ker\pi'$, so $\bar{s}_{1}\cdots\bar{s}_{q}=1$ is a relation of $Q$. Since $\bar{R}_{Q}$ normally generates all relations of $Q$ in $F(\bar{S}_{Q})$:
		$$\bar{s}_{1}\cdots \bar{s}_{q}=\bar{x}_{1}\bar{r}_{1}\bar{x}_{1}^{-1}\bar{x}_{2}\bar{r}_{2}\bar{x}_{2}^{-1}\cdots\bar{x}_{t}\bar{r}_{t}\bar{x}_{t}^{-1}$$
	where $\bar{x}_{i}\in F(\bar{S}_{Q})$ and $\bar{r}_{i}\in\bar{R}_{Q}$ for all $i$. For each $i$, $\bar{x}_{i}$ lifts to a word $x_{i}\in F(S_{Q})$ and $\bar{r}_{i}$ lifts to the relation $r_{i}=k_{i}$ where $k_{i}\in K$. Therefore $w_{1}$ is equivalent to the following in $G'$:
		$$x_{1}k_{1}x_{1}^{-1}x_{2}k_{2}x_{2}^{-1}\cdots x_{t}k_{t}x_{t}^{-1}$$
	Furthermore, for each $i$, $k_{i}$ can be represented by $k_{i,1}k_{i,2}\cdots k_{i,\ell_{i}}$ where each $k_{i,j}\in S_{K}$. Therefore, for each $i$, $x_{i}k_{i}x_{i}^{-1}$ can be represented by:
		$$x_{i}k_{i,1}x_{i}^{-1}x_{i}k_{i,2}x_{i}^{-1}\cdots x_{i}k_{i,\ell_{i}}x_{i}^{-1}$$
	Since each $x_{i}\in F(S_{Q})$, applying relations of $R'$ to $x_{i}k_{i,j}x_{i}^{-1}$ yields $x_{i}k_{i,j}x_{i}^{-1}=k_{i,j}'$ for some $k_{i,j}'\in K$. For each $i$ we get $k_{i}=k_{i,1}k_{i,2}\cdots k_{i,\ell_{i}}=k_{i}'$ for some $k_{i}'\in K$. Therefore $g=k_{1}'k_{2}'\cdots k_{t}'w_{2}$ where $k_{i}'$ and $w_{2}$ are words in $F(S_{K})$ for all $i$. Therefore $g$ can be expressed as a word in $F(S_{K})$ and $g\in K$. Thus $\ker\pi'\subseteq \im\iota'$ and the top row of \eqref{commuting diagram for presentation} is exact. 
\end{proof}
\begin{theorem}\label{thm: building a presentation}
	Let $K,G$, and $Q$ be defined above with $\mathcal{S}$ and $R$ as in the paragraphs on generators and relations. Then $\langle \mathcal{S}\mid R\rangle$ is a presentation for $G$.
\end{theorem} 
\begin{proof}
	Consider the diagram \eqref{commuting diagram for presentation}:
		$$\begin{tikzcd}
			1\arrow{r}&K\arrow{r}{\iota'}\arrow{d}{\id_{K}}&G'\arrow{r}{\pi'}\arrow{d}{f}&Q\arrow{r}\arrow{d}{\id_{Q}}&1\\
			1\arrow{r}&K\arrow{r}{\iota}&G\arrow{r}{\pi}&Q\arrow{r}&1
		\end{tikzcd}$$
	with $\iota'$ and $\pi'$ defined as in Lemmas \ref{lemma: iota' is injective} and \ref{lemma: def of pi' and surjective}. It suffices to show $f:G'\to G$ is an isomorphism. Since $G$ is an extension of $Q$ by $K$, the bottom row of \eqref{commuting diagram for presentation} is exact. By Theorem \ref{thm: presentation is exact}, the top row of \eqref{commuting diagram for presentation} is exact. Therefore both rows of \eqref{commuting diagram for presentation} are exact. Furthermore, $f$ is a natural map so $f$ makes the diagram commute. Thus $f$ is an isomorphism by the 5-Lemma. 
\end{proof}
	
\subsection{Group cohomology}\label{group cohomology}
In this section we will review the general group cohomology required to classify group extensions by 2-cocycles. We begin with constructing the normalized bar resolution, then give the definition of equivalent group extensions, and describe the construction of a cocycle from an extension. The results of this section are known and can be found more thoroughly in chapters I and IV of Brown's text \cite{Brown}.
\paragraph{Normalized standard resolution}
Let $G$ be a group and let $P_{t}$ be the free $\mathbb{Z}$ module generated by $t+1$ tuples $(p_{0},\ldots p_{t})$ where $p_{i}\in G$ for all $i$. $G$ acts on $(p_{0},\ldots, p_{n})$ by $p\cdot(p_{0},\ldots,p_{n})=(p\cdot p_{0},\ldots,p\cdot p_{t})$ for any $p\in G$. To construct a chain complex, we use the boundary operator $\partial_{t}^{P}:P_{t}\to P_{t-1}$ determined by $\partial_{t}^{P}=\sum\limits_{i=0}^{t}(-1)^{i}d_{i}$ where: $$d_{i}(p_{0},\ldots, p_{t})=(p_{0},\ldots, p_{i-1},p_{i+1},\ldots, p_{t})$$
As $\mathbb{Z}G$ modules, $P_{t}$ is freely generated by elements $(1,p_{1},\ldots,p_{t})$ which represent the $G$ orbits of the $t+1$ tuples where $p_{0}=1$. We use the following bar notation to represent elements of $P_{t}$ as $G$-orbits:
$[p_{1}\mid p_{2}\mid \ldots\mid p_{t}]=(1,p_{1},p_{1}p_{2},\ldots p_{1}p_{2}\cdots p_{t})$. The change of basis results in the following change to $d_{i}$ in the boundary operator:
	$$d_{i}[p_{1}\mid\cdots\mid p_{t}]=
	\begin{cases}
		p_{1}[p_{2}\mid\cdots\mid p_{t}]\hspace{1em}&i=0\\
		[p_{1}\mid\cdots\mid p_{i-1}\mid p_{i}p_{i+1}\mid p_{i+2}\mid\cdots\mid p_{t}]\hspace{1em}&0<i<t\\
		[p_{1}\mid\cdots\mid p_{t-1}]\hspace{1em}&i=t
	\end{cases}$$
In particular $\partial_{2}^{P}([p_{1}\mid p_{2}])=p_{1}[p_{2}]-[p_{1}p_{2}]+[p_{1}]$ and  $\partial_{1}^{P}([p_{1}])=(p_{1}-1)[\hspace{1em}]$. The augmentation map $\varepsilon_{P}:P_{0}\to\mathbb{Z}$ defined by $\varepsilon_{P}([\hspace{1em}])=1$ implies
	$$\begin{tikzcd}
		\mathcal{P}:&\cdots\arrow{r}{\partial_{3}^{P}}&P_{2}\arrow{r}{\partial_{2}^{P}}&P_{1}\arrow{r}{\partial_{1}^{P}}&P_{0}\arrow{r}{\varepsilon_{P}}&\mathbb{Z}\arrow{r}&0
	\end{tikzcd}$$
is a free resolution of $\mathbb{Z}$ over $\mathbb{Z}G$ modules. Now, let $D_{t}$ be the subcomplex of $P_{t}$ generated over $\mathbb{Z}G$ by elements $[p_{1}\mid\ldots\mid p_{t}]$ such that $p_{i}=1$ for some $i$. Then $\bar{P}_{t}=P_{t}/D_{t}$ with the maps of the chain complex $\mathcal{\bar{P}}$ induced from $\mathcal{P}$ defines a free resolution of $\mathbb{Z}$ over $\mathbb{Z}G$ modules, called the \textit{normalized standard resolution}.\par
Furthermore, two projective resolutions of $\mathbb{Z}$ over $\mathbb{Z}G$ are chain homotopy equivalent. For a proof of this fact the interested reader is referred to chapter one of Brown's text \cite{Brown}. 
\paragraph{Equivalent extensions}
Let $K$ be an abelian group and let $Q$ be a group which acts on $K$ by the map $\theta:Q\to Aut(K)$. An extension of $Q$ by $K$ giving rise to $\theta$ is a short exact sequence:
	$$\begin{tikzcd}
		1\arrow{r}&K\arrow{r}{\iota}&E\arrow{r}{\pi}&Q\arrow{r}&1
	\end{tikzcd}$$
with the following condition: for any $k\in K$ and $\tilde{g}\in E$ such that $\pi(\tilde{g})=g\in G$, then $\tilde{g}\iota(k)\tilde{g}^{-1}=\iota(\theta(g)(k))$. Two group extensions $E_{1}$ and $E_{2}$ are equivalent if there exists an isomorphism $\varphi:E_{1}\to E_{2}$ such that the diagram:
	$$\begin{tikzcd}
		&&E_{1}\arrow{dr}\arrow{dd}{\varphi}&&\\
		1\arrow{r}&K\arrow{ur}\arrow{dr}&&Q\arrow{r}&1\\
		&&E_{2}\arrow{ru}&&
	\end{tikzcd}$$
commutes. Let $\mathcal{E}(Q;K)$ denote the set of equivalence classes of these extensions.
\paragraph{Constructing 2-cocycles}
Suppose $K$ is an abelian group and $Q$ acts on $K$ by the action of $\theta$ as in the preceding paragraph. Consider the extension:
	$$\begin{tikzcd}
		0\arrow{r}&K\arrow{r}{\iota}&E\arrow{r}{\pi}&Q\arrow{r}&1
	\end{tikzcd}$$
A section, $s$, is a function $s:Q\to E$ such that $\pi\circ s=\id_{Q}$; furthermore $s$ is \textit{normalized} if $s(1_{Q})=1_{E}$. Since $s$ is a set theoretical function, $s(p_{1})s(p_{2})$ is not necessarily equal to $s(p_{1}p_{2})$ in $E$. Therefore we can measure the failure of $s$ to be a homomorphism by a function $\kappa\in\hom_{Q}(P_{2},K)$ such that $s(p_{1})s(p_{2})=\iota(\kappa([p_{1}\mid p_{2}]))s(p_{1}p_{2})$. Hence we have a formula:
	$$\kappa([p_{1}\mid p_{2}])=s(p_{1})s(p_{2})s(p_{1}p_{2})^{-1}$$
which determines a function corresponding to $E$ as an extension of $Q$ by $K$. A thorough explanation that $\kappa$ satisfies the cocycle condition can be found in chapter IV, section 3 of Brown \cite{Brown}.
\paragraph{Corresponding Extensions}
Let $\kappa$ be a representative for a cohomology class in $H^{2}(Q;K)$ determined by the normalized standard resolution. Suppose $Q$ acts on $K$ by the action $\theta$, define $E_{\kappa}$ to be the twisted semi-direct product $K\ltimes_{\kappa}Q$ with multiplication defined by:
$$(a,g)\cdot(b,h)=(a+g\cdot b+\kappa(g,h),gh)$$
Note that multiplication in $E_{\kappa}$ satisfies associativity since $\kappa$ is a cocycle (page 92 of \cite{Brown}). Thus $E_{\kappa}$ is a representative for the equivalence class of group extensions of $Q$ by $K$ corresponding to $[\kappa]$.

The following theorem by Eilenberg and MacLane provides the classification of group extensions by 2-cocycles constructed above \cite{E/M}.
\begin{theorem}[Eilenberg MacLane 1947]\label{Classification of groups by cocycles}
	Suppose $Q$ and $K$ are groups with $K$ abelian such that $Q$ acts on $K$ by $\theta$. There exists a bijection between $\mathcal{E}(Q,K)$ and $H^{2}(Q;K)$.
\end{theorem}
The choice of section determines the representative of the cohomology class in $H^{2}(Q;K)$. Furthermore, if the extension, $E$, is split, then $E\approx K\rtimes_{\theta} Q$ corresponds to the cohomology class represented by the trivial cocycle. Changing the choice of projective resolution of $\mathbb{Z}$ over $\mathbb{Z}G$ yields a corresponding representative in an isomorphic cohomology group.
\subsection{Braid groups and symmetric groups}
\paragraph{Symmetric Group} The standard presentation for the symmetric group generated by adjacent transpositions is:
	\begin{equation}\label{symmetric group generated by adjacent transpositions}
		S_{n}=\left\langle \sigma_{1},\ldots,\sigma_{n-1}\middle\mid\hspace{1em}\begin{aligned}
			&\sigma_{i}^{2}=1,\hspace{.5em}[\sigma_{i},\sigma_{j}]=1\text{ if }|i-j|>1\hspace{1em}\\
			&\sigma_{i}\sigma_{i+1}\sigma_{i}=\sigma_{i+1}\sigma_{i}\sigma_{i+1}
		\end{aligned}\right\rangle
	\end{equation}
Let $\sigma_{i,j}$ represent the transposition which permutes $i$ and $j$. As an element of \eqref{symmetric group generated by adjacent transpositions} we take the convention: $$\sigma_{i,j}=\sigma_{i}\cdots\sigma_{j-2}\sigma_{j-1}\sigma_{j-2}^{-1}\cdots\sigma_{i}^{-1}$$
where $i<j$. Throughout this paper, we will use the following presentation for the symmetric group:
	\begin{equation}\label{symmetric group}
		S_{n}=\left\langle \{\sigma_{i,j}\}_{1\leq i<j\leq n}\middle|\hspace{1em}
		\begin{aligned}&\sigma_{i,j}^{2}=1, \hspace{1em} [\sigma_{i,j},\sigma_{k,\ell}]=1\text{ if }\{i,j\}\cap\{k,\ell\}=\emptyset\hspace{1em}\\
			&\sigma_{i,j}\sigma_{j,k}\sigma_{i,j}^{-1}=\sigma_{i,k}\text{ for all }i,j,k
		\end{aligned}\right\rangle
	\end{equation}
\paragraph{Braid groups}
Let $x_{1},\ldots, x_{n}$ be marked points in $\mathbb{C}$. Elements of $B_{n}$ can be represented by a collection of $n$ non-colliding paths $f_{i}:[0,1]\to\mathbb{C}\times [0,1]$ such that $f_{i}(0),f_{i}(1)\in\{x_{1},\ldots,x_{n}\}$ and $f_{i}(t)\in\mathbb{C}\times\{t\}$. Then the collection $\{f_{i}\}$ determines a permutation on $\{1,\ldots,n\}$. Throughout this paper we will represent elements of $B_{n}$ by strand diagrams representing the paths $f_{1}(t),\ldots,f_{n}(t)$. A positively oriented twist between strands $i$ and $i+1$, denoted by $b_{i}$, corresponds to a clockwise twist in the strand diagram, where the $(i+1)^{st}$ strand passes over the $i^{th}$ strand.
	
For $1\leq i<j\leq n$, we define the half twist between the $i$ and $j$ strands by:
	$$b_{i,j}=b_{i}\cdots b_{j-2}b_{j-1}b_{j-2}^{-1}\cdots b_{i}^{-1}$$
Under the strand diagrams, this is equivalent to pulling all strands between the $i^{th}$ and $j^{th}$ strands over the $i^{th}$ strand, half twist the $i^{th}$ and $j^{th}$ strands by pulling the $j^{th}$ strand over the $i^{th}$ strand, then pull all strands between $i$ and $j$ over the $j^{th}$ strand. This choice of $b_{i,j}$ yields the following presentation for $B_{n}$ which is equivalent to the Birman-Ko-Lee presentation \cite{BKL}:
	\begin{equation}\label{braid group}
		B_{n}\approx\left\langle \{b_{i,j}\}_{1\leq i<j\leq n}\middle| \begin{aligned}
			\hspace{.5em}&[b_{i,j},b_{k,\ell}]=1\text{ if }(j-k)(j-\ell)(i-k)(i-\ell)>0\\
			&b_{i,j}b_{j,k}b_{i,j}^{-1}=b_{i,k}\text{ if }i<j<k,\hspace{.25em}k<i<j,\hspace{.25em}j<k<i\\
			&b_{i,j}^{-1}b_{j,k}b_{i,j}=b_{i,k}\text{ if }j<i<k,\hspace{.25em}i<k<j,\hspace{.25em}k<j<i
		\end{aligned}\hspace{.5em}\right\rangle
	\end{equation}
Note that the third relation can be determined by the second relation. Conjugating the second relation by $b_{i,j}^{-1}$ and renaming the indices provides the third relation. However the third relation is included for clarity in later computations.
\paragraph{Pure braids}
As a subgroup of the braid group, elements of the pure braid group, $PB_{n}$, are braids in which the induced permutation on $\{1,\ldots,n\}$ is trivial. Considering strand diagrams, every strand begins and ends at the same point of $\mathbb{C}$. In terms of \eqref{braid group}, $PB_{n}$ is generated by all $b_{i,j}^{2}$. For any pure braid, we can define the winding number of the $i^{\text{th}}$ and $j^{\text{th}}$ strand to be the number of positively oriented full twists between those two strands.
\paragraph{Level $m$ braid group}
Let $\mathbb{Z}_{m}$ denote $\mathbb{Z}/m\mathbb{Z}$ and $B_{n}$ be the braid group. Evaluating the unreduced Burau representation (See Birman \cite{Birman}), at $t=-1$, and reducing$\mod m$ (for any $m\geq0$) yields the following map from $B_{n}$ to $GL_{n}(\mathbb{Z}_{m})$:
	$$\begin{tikzcd}
		B_{n}\arrow{r}{\rho}& GL_{n}(\mathbb{Z}[t,t^{-1}])\arrow{r}{t=-1}& GL_{n}(\mathbb{Z})\arrow{r}& GL_{n}(\mathbb{Z}_{m})
	\end{tikzcd}$$
The map $\rho:B_{n}\to GL_{n}(\mathbb{Z}[t,t^{-1}])$ is the unreduced Burau representation of the braid group defined on the generators of $\eqref{Artin braid}$ by:
	$$\rho(b_{i})\mapsto I_{i-1}\oplus \begin{pmatrix}1-t&t\\1&0\end{pmatrix}\oplus I_{n-i-1}$$
For $m\geq 0$, the \textit{level m braid group} $B_{n}[m]$ is defined as the kernel of this composition, which is a finite index subgroup of $B_{n}$. A thorough topological description of the Burau representation and alternative descriptions of $B_{n}[m]$ can be found in Brendle and Margalit's paper \cite{B/M}.

\begin{definition} The \textit{mod 4 braid group}, denoted $\Z_{n}$, is the quotient of the braid group by the level 4 congruence subgroup of the braid group, $B_{n}/B_{n}[4]$.
\end{definition}

For arbitrary $m$, little is known about the algebraic structure of $B_{n}[m]$. In the case $m=4$, Brendle and Margalit proved $B_{n}[4]=PB_{n}^{2}$, the subgroup of $PB_{n}$ generated by squares of all elements \cite{B/M}. The algebraic structure of quotients of level $m$ braid groups is better understood. Stylianakis proved that for each odd prime $p$, $B_{n}[p]/B_{n}[2p]\approx S_{n}$ \cite{Styl}. Appel, Bloomquist, Gravel, and Holden generalized Stylianakis' result to $B_{n}[\ell]/B_{n}[2\ell]\approx S_{n}$ for every odd, positive integer $\ell$ \cite{A/B/G/H}. In the same work they also proved the following Theorem, which yields greater context for the group presentation of $\Z_{n}$ we give \cite{A/B/G/H}:
\begin{theorem}[Appel, Bloomquist, Gravel, Holden]
	For each $n$ and each $\ell$ odd:
	$$B_{n}[\ell]/B_{n}[4\ell]\approx\Z_{n}$$
\end{theorem}
\paragraph{Formalization of Theorem \ref{main theorem}:}Define $\mathcal{Z}_{n}=B_{n}/PB_{n}^{2}$ and $\mathcal{PZ}_{n}$ be the image of $PB_{n}$ in $\mathcal{Z}_{n}$. The standard surjection of $B_{n}$ onto the symmetric group $S_{n}$ yields the following non split group extension \cite{K/M}:
	$$1\to\mathcal{PZ}_{n}\to \mathcal{Z}_{n}\to S_{n}\to 1$$
Kordek and Margalit also proved $\mathcal{PZ}_{n}\approx\mathbb{Z}_{2}^{n\choose 2}$ where the action of $S_{n}$ on $\mathcal{PZ}_{n}$, represented by $\theta$, is induced by the action of $B_{n}$ on $PB_{n}$ \cite{K/M}. Let $[\kappa]\in H^{2}(S_{n};\mathcal{PZ}_{n})$ be the nontrivial cohomology class corresponding to this extension. Since $\mathbb{Z}^{n\choose 2}\approx H_{1}(PB_{n};\mathbb{Z})$ is the abelianization of $PB_{n}$ (page 252 of \cite{F/M}), consider the group extension:
	$$0\to H_{1}(PB_{n};\mathbb{Z})\to G_{n}\to S_{n}\to 1$$
where $G_{n}$ is the quotient of $B_{n}$ by the commutator subgroup of $PB_{n}$. Let $[\phi]\in H^{2}(S_{n};H_{1}(PB_{n};\mathbb{Z}))$ be the cohomology class corresponding to the extension of $S_{n}$ by $H_{1}(PB_{n};\mathbb{Z})$. In this paper we prove a representative of $[\phi]$ composed with the$\mod 2$ reduction of integers determines a representative for $[\kappa]$. Furthermore, we show $[\phi]$ is order $2$ in $H^{2}(PB_{n};H_{1}(PB_{n};\mathbb{Z}))$. 
	
\section{The universal cover of a $\mathcal{K}(S_{n},1)$-space}
In this section we construct a truncated resolution of $\mathbb{Z}$ over $\mathbb{Z}S_{n}$ corresponding to the universal cover of a $\mathcal{K}(S_{n},1)$-space. We approximate the cellular chain complex in dimensions 0,1 and 2 for the universal cover, $\tilde{X}$, of the 2-skeleton for the Cayley complex of $S_{n}$. 
Let $\mathcal{R}$ be the chain complex of $\tilde{X}$ and recall $\mathcal{P}$ is the normalized bar resolution. Since each dimension of $\mathcal{R}$ is a free $S_{n}$ module and $\mathcal{P}$ is acyclic, for each $i$ there exists $\gamma_{i}:R_{i}\to P_{i}$ which fits into the following commutative diagram:
	$$\begin{tikzcd}
		\cdots\arrow{r}{\partial_{3}^{R}}&R_{2}\arrow{r}{\partial_{2}^{R}}\arrow{d}{\gamma_{2}}&R_{1}\arrow{r}{\partial_{1}^{R}}\arrow{d}{\gamma_{1}}&R_{0}\arrow{r}{\varepsilon_{R}}\arrow{d}{\gamma_{0}}&\mathbb{Z}\arrow{r}\arrow{d}{id}&0\\
		\cdots\arrow{r}{\partial_{3}^{P}}&\bar{P}_{2}\arrow{r}{\partial_{2}^{P}}&\bar{P}_{1}\arrow{r}{\partial_{1}^{P}}&\bar{P}_{0}\arrow{r}{\varepsilon_{P}}&\mathbb{Z}\arrow{r}&0
	\end{tikzcd}$$
The existence of $\gamma_{i}$ yields the following theorem:
\begin{theorem}\label{comp of cocycles}
	Let $0\to K\to E\to S_{n}\to 1$ be a group extension which corresponds to $[\kappa]\in H^{2}(S_{n};K)$ under the normalized standard resolution. There exists $\kappa'\in\hom(R_{2},K)$ such that $[\kappa']=[\kappa]$ in $H^{2}(S_{n};K)$ defined by:
		$$\kappa'=\kappa\circ\gamma_{2}$$
\end{theorem}
\paragraph{Cayley Complex}To describe the universal cover of a $\mathcal{K}(S_{n},1)$-space, we use the group presentation of the symmetric group generated by all possible transpositions given by \eqref{symmetric group}. Let $X$ be a $\mathcal{K}(S_{n},1)$-space which has the presentation complex as its 2-skeleton. Then $X$ has one 0-cell, a 1-cell for each generator of $S_{n}$, and a 2-cell for each relation in $S_{n}$. Let $x_{0}$ denote the 0-cell of $X$ and $x_{i,j}$ denote the 1-cell of $X$ corresponding to the generator $\sigma_{i,j}$ of $S_{n}$. Note that we only use positive generators to label 1-cells. Let $c_{i,j}$ be the two cell glued by the relation $\sigma_{i,j}^{2}=1$ and $d_{i,j,k,\ell}$ be the two cell glued by the relation $[\sigma_{i,j},\sigma_{k,\ell}]=1$ if $\{i,j\}\cap\{k,\ell\}=\emptyset$. 
	\begin{figure*}
		\caption{Gluing of the Cayley Complex}\label{figure:Gluing Maps}
		$$\begin{tikzpicture}
			\draw[-] (-1,-1) to (-1,1) to (1,1) to (1,-1) to (-1,-1);
			\draw[-] (4,-1) to (4,1) to (6,1) to (6,-1) to (4,-1);
			\node[circle, fill, scale=.25] at (-1,1){};
			\node[circle, fill, scale=.25] at (1,1){};
			\node[circle, fill, scale=.25] at (1,-1){};
			\node[circle, fill, scale=.25] at (-1,-1){};
			\node[circle, fill, scale=.25] at (4,1){};
			\node[circle, fill, scale=.25] at (6,1){};
			\node[circle, fill, scale=.25] at (4,-1){};
			\node[circle, fill, scale=.25] at (6,-1){};
			\draw[->] (-1,-1) to (0,-1);
			\draw[->] (1,-1) to (1,0);
			\draw[->] (-1,1) to (0,1);
			\draw[->] (-1,-1) to (-1,0);
			\draw[->] (4,-1) to (5,-1);
			\draw[->] (6,-1) to (6,0);
			\draw[->] (4,1) to (5,1);
			\draw[->] (4,-1) to (4,0);
			\node[below left, scale=.75] at (-1,-1){$\tilde{x}_{0}$};
			\node[below right, scale=.75] at (1,-1){$\sigma_{i,j}\tilde{x}_{0}$};
			\node[above right, scale=.75] at (1,1){$\sigma_{i,j}\sigma_{k,\ell}\tilde{x}_{0}$};
			\node[above left, scale=.75] at (-1,1){$\sigma_{k,\ell}\tilde{x}_{0}$};
			\node[below left, scale=.75] at (4,-1){$\tilde{x}_{0}$};
			\node[below right, scale=.75] at (6,-1){$\sigma_{i,j}\tilde{x}_{0}$};
			\node[above right, scale=.75] at (6,1){$\sigma_{i,j}\sigma_{j,k}\tilde{x}_{0}$};
			\node[above left, scale=.75] at (4,1){$\sigma_{i,k}\tilde{x}_{0}$};
			\node[below, scale=.75] at (0,-1.15){$\tilde{x}_{i,j}$};
			\node[right, scale=.75] at (1.15,0){$\sigma_{i,j}\tilde{x}_{k,\ell}$};
			\node[above, scale=.75] at (0,1.15){$\sigma_{k,\ell}\tilde{x}_{i,j}$};
			\node[left, scale=.75] at (-1.15,0){$\tilde{x}_{k,\ell}$};
			\node[scale=.75] at (0,0){$\tilde{d}_{i,j,k,\ell}$};
			\draw[->] (0,-.5) arc (-90:180:.5cm);
			\node[below, scale=.75] at (5,-1.15){$\tilde{x}_{i,j}$};
			\node[right, scale=.75] at (6.15,0){$\sigma_{i,j}\tilde{x}_{j,k}$};
			\node[above, scale=.75] at (5,1.15){$\sigma_{i,k}\tilde{x}_{i,j}$};
			\node[left, scale=.75] at (3.85,0){$\tilde{x}_{i,k}$};
			\node[scale=.75] at (5,0){$\tilde{e}_{i,j,k}$};
			\draw[->] (5,-.5) arc (-90:180:.5cm);
		\end{tikzpicture}$$
	\end{figure*}
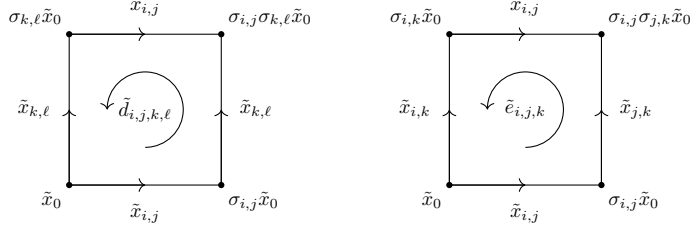
Consider the relation $\sigma_{i,j}\sigma_{j,k}\sigma_{i,j}^{-1}=\sigma_{i,k}$ from \eqref{symmetric group} as  $\sigma_{i,j}\sigma_{j,k}\sigma_{i,j}^{-1}\sigma_{i,k}^{-1}=1$; let $e_{i,k,j}$ be the 2-cell that glues along this relation in $X$. Note we have one 2-cell $e_{i,j,k}$ for each distinct, ordered triple in $\{1,2,\ldots,n\}$ while $c_{i,j}$ and $d_{i,j,k,\ell}$ assume $i<j$ and $k<\ell$.
	
Let $\tilde{X}$ be the universal cover of $X$ with the basepoint, $\tilde{x}_{0}$, chosen to be the lift of $x_{0}$. For each $t$-cell of $X$, there are $n!$ choices of lifts in $\tilde{X}$. Then the $0$-cells of $\tilde{X}$ are of the form $g\cdot \tilde{x}_{0}$ for $g\in S_{n}$ (with $g=1_{S_{n}}$ representing $\tilde{x}_{0}$) and $R_{0}$ is generated as a $\mathbb{Z}S_{n}$ module by $\tilde{x}_{0}$.

Now, lift $x_{i,j}$ to the $1$-cell, $\tilde{x}_{i,j}$, of $\tilde{X}$ which begins at $\tilde{x}_{0}$ and ends at $\sigma_{i,j}\cdot \tilde{x}_{0}$. Notice that for each $0$-cell, $g\cdot \tilde{x}_{0}$, there is a $1$-cell, denoted $g\cdot\tilde{x}_{i,j}$, beginning at $g\cdot\tilde{x}_{0}$ and ending at $g\sigma_{i,j}\cdot\tilde{x}_{0}$. Therefore $R_{1}$ is generated as a $\mathbb{Z}S_{n}$ module by the set of $\tilde{x}_{i,j}$ where $1\leq i<j\leq n$. Furthermore, the gluing map determines the differential $\partial_{1}^{R}:R_{1}\to R_{0}$ by $\partial_{1}^{R}(\tilde{x}_{i,j})=(\sigma_{i,j}-1)\tilde{x}_{0}$.

Let $\tilde{c}_{i,j}$ denote the lift of $c_{i,j}$ glued into $\tilde{X}$ along the loop which begins at $\tilde{x}_{0}$, follows $\tilde{x}_{i,j}$ to $\sigma_{i,j}\cdot \tilde{x}_{0}$, then follows $\sigma_{i,j}\cdot\tilde{x}_{i,j}$ back to $\tilde{x}_{0}$. Define $\tilde{d}_{i,j,k,\ell}$ to be the lift of $d_{i,j,k,\ell}$ glued into $\tilde{X}$ by the loop which starts at $\tilde{x}_{0}$, follows $\tilde{x}_{i,j}$ to $\sigma_{i,j}\cdot\tilde{x}_{0}$, then follows $\sigma_{i,j}\cdot\tilde{x}_{k,\ell}$ to $\sigma_{i,j}\sigma_{k,\ell}\cdot\tilde{x}_{0}$, then follows $\sigma_{k,\ell}\cdot\tilde{x}_{i,j}$ in reverse to $\sigma_{k,\ell}\cdot\tilde{x}_{0}$, and follows $\tilde{x}_{k,\ell}$ in reverse back to $\tilde{x}_{0}$. Furthermore, let $\tilde{e}_{i,j,k}$ denote the lift of $e_{i,j,k}$ glued into $\tilde{X}$ by the loop which begins at $\tilde{x}_{0}$, follows $\tilde{x}_{i,j}$ to $\sigma_{i,j}\cdot\tilde{x}_{0}$, then follows $\sigma_{i,j}\cdot\tilde{x}_{j,k}$ to $\sigma_{i,j}\sigma_{j,k}\cdot\tilde{x}_{0}$, then follows $\sigma_{i,k}\cdot\tilde{x}_{i,j}$ in reverse to $\sigma_{i,k}\cdot\tilde{x}_{0}$, and finally follows $\tilde{x}_{i,k}$ in reverse back to $\tilde{x}_{0}$. Notice the paths of $\tilde{d}_{i,j,k,\ell}$ and $\tilde{e}_{i,j,k}$ are given in Figure \ref{figure:Gluing Maps}.

Furthermore, for each $g\cdot\tilde{x}_{0}$, there is a lift of $c_{i,j}$, $d_{i,j,k,\ell}$, and $e_{i,j,k}$ glued into $\tilde{X}$ by a corresponding loop which begins and ends at $g\cdot\tilde{x}_{0}$, denoted $g\cdot\tilde{c}_{i,j}$, $g\cdot\tilde{d}_{i,j,k,\ell}$, and $g\cdot\tilde{e}_{i,j,k}$ respectively. Therefore, as a free $\mathbb{Z}S_{n}$ module, $R_{2}$ is generated by the set of all $\tilde{c}_{i,j}$, $\tilde{d}_{i,j,k,\ell}$, and $\tilde{e}_{i,j,k}$. To determine $\partial_{2}^{R}:R_{2}\to R_{1}$, we use the gluing of $2$-cells in $\tilde{X}$. Therefore $\partial_{2}^{R}(\tilde{c}_{i,j})=(\sigma_{i,j}+1)\tilde{x}_{i,j}$ and by Figure \ref{figure:Gluing Maps} we have: 
	\begin{align*}
		\partial_{2}^{R}(\tilde{d}_{i,j,k,\ell})=\tilde{x}_{i,j}+\sigma_{i,j}\cdot\tilde{x}_{k,\ell}-\sigma_{k,\ell}\cdot\tilde{x}_{i,j}-\tilde{x}_{k,\ell}\\
		\partial_{2}^{R}(\tilde{e}_{i,k,j})=\tilde{x}_{i,j}+\sigma_{i,j}\cdot\tilde{x}_{j,k}-\sigma_{i,k}\cdot\tilde{x}_{i,j}-\tilde{x}_{i,k}
	\end{align*}
Therefore we have proven the following lemma.
	\begin{lemma}\label{generators of cayley complex}
		As a $\mathbb{Z}S_{n}$ modules, $R_{1}$ is generated by $\{\tilde{x}_{i,j}\}$ and $R_{0}$ by $\tilde{x}_{0}$. Furthermore $R_{2}$ is generated by $\{\tilde{c}_{i,j}\}\cup\{\tilde{d}_{i,j,k,\ell}\}\cup\{\tilde{e}_{i,j,k}\}$ where $i<j$ and $k<\ell$.
	\end{lemma}
By Lemma \ref{generators of cayley complex}, defining $\gamma_{0}$ and $\gamma_{1}$ by $\gamma_{0}(\tilde{x}_{0})=[\hspace{1em}]$ and $\gamma_{1}(\tilde{x}_{i,j})=[\sigma_{i,j}]$ implies $\gamma_{0}$ and $\gamma_{1}$ commute with the differential.
\begin{theorem}\label{thm: chain map}
	Define $\gamma_{2}:R_{2}\to \bar{P}_{2}$ by:
	\begin{align*}		
		\gamma_{2}(\tilde{c}_{i,j})&=[\sigma_{i,j}\mid\sigma_{i,j}]\\
		\gamma_{2}(\tilde{d}_{i,j,k,\ell})&=[\sigma_{i,j}\mid\sigma_{k,\ell}]-[\sigma_{k,\ell}\mid\sigma_{i,j}]\\
		\gamma_{2}(\tilde{e}_{i,k,j})&=[\sigma_{i,j}\mid \sigma_{j,k}]-[\sigma_{i,k}\mid \sigma_{i,j} ]
	\end{align*}
	Then $\gamma_{2}$ commutes with the differentials:
	$$\begin{tikzcd}
		R_{2}\arrow{r}{\partial_{2}^{R}}\arrow{d}{\gamma_{2}}&R_{1}\arrow{d}{\gamma_{1}}\arrow{r}{\partial_{1}^{R}}&R_{0}\arrow{d}{\gamma_{0}}\arrow{r}{\epsilon_{R}}&\mathbb{Z}\arrow{d}{\id}\arrow{r}&0\\
		\bar{P}_{2}\arrow{r}{\partial_{2}^{P}}&\bar{P}_{1}\arrow{r}{\partial_{1}}&\bar{P}_{0}\arrow{r}{\epsilon_{P}}&\mathbb{Z}\arrow{r}&0
	\end{tikzcd}$$
\end{theorem}
\begin{proof}
	$\gamma_{2}$ is a $\mathbb{Z}S_{n}$ module homomorphism defined on generators. Therefore it suffices to show $\gamma_{1}\circ\partial_{2}^{R}=\partial_{2}^{P}\circ\gamma_{2}$. Recall the generators of $R_{1}$ are $\tilde{x}_{i,j}$ and $\gamma_{1}(\tilde{x}_{i,j})=[\sigma_{i,j}]$. By following the gluing maps in Figure \ref{figure:Gluing Maps} we have: 
	\begin{align*}
		\gamma_{1}\circ\partial_{2}^{R}(\tilde{c}_{i,j})&=[\sigma_{i,j}]+\sigma_{i,j}[\sigma_{i,j}]\\
		\gamma_{1}\circ\partial_{2}^{R}(\tilde{d}_{i,j,k,\ell})&=[\sigma_{i,j}]+\sigma_{i,j}[\sigma_{k,\ell}]-\sigma_{k,\ell}[\sigma_{i,j}]-[\sigma_{k,\ell}]\\
		\gamma_{2}\circ\partial_{2}^{R}(\tilde{e}_{i,k,j})&=[\sigma_{i,j}]+\sigma_{i,j}[\sigma_{j,k}]-\sigma_{i,k}[\sigma_{i,j}]-[\sigma_{i,k}]
	\end{align*}
	Recall we are using the normalized bar resolution, therefore $[p_{1}\mid\cdots \mid p_{t}]=0$ if $p_{i}$ is trivial for any $i$. Now, since $\partial_{2}^{R}([p_{1}\mid p_{2}])=p_{1}[p_{2}]-[p_{1}\cdot p_{2}]+[p_{1}]$ for any $p_{1},p_{2}\in S_{n}$ and $\partial_{2}^{P}$ is a $\mathbb{Z}S_{n}$ module homomorphism, we have:
	\begin{align*}
			\partial_{2}^{P}([\sigma_{i,j}\mid\sigma_{i,j}])&=\sigma_{i,j}[\sigma_{i,j}]-[\sigma_{i,j}^{2}]+[\sigma_{i,j}]\\
			&=\sigma_{i,j}[\sigma_{i,j}]-[1]+[\sigma_{i,j}]\\
			&=[\sigma_{i,j}]+\sigma_{i,j}[\sigma_{i,j}]
	\end{align*}
	Thus $(\partial_{2}^{P}\circ\gamma_{2})(\tilde{c}_{i,j})=(\gamma_{1}\circ\partial_{2}^{R})(\tilde{c}_{i,j})$. To compute $\partial_{2}^{P}(\gamma_{2}(\tilde{d}_{i,j,k,\ell}))$, recall that $\{i,j\}\cap\{k,\ell\}=\emptyset$, therefore $\sigma_{i,j}\sigma_{k,\ell}=\sigma_{k,\ell}\sigma_{i,j}$:
	\begin{align*}
		\partial_{2}^{P}([\sigma_{i,j}\mid\sigma_{k,\ell}]-[\sigma_{k,\ell}\mid\sigma_{i,j}])&=\sigma_{i,j}[\sigma_{k,\ell}]-[\sigma_{i,j}\sigma_{k,\ell}]+[\sigma_{i,j}]-\sigma_{k,\ell}[\sigma_{i,j}]+[\sigma_{k,\ell}\sigma_{i,j}]-[\sigma_{k,\ell}]\\
		&=\sigma_{i,j}[\sigma_{k,\ell}]-[\sigma_{i,j}\sigma_{k,\ell}]+[\sigma_{i,j}]-\sigma_{k,\ell}[\sigma_{i,j}]+[\sigma_{i,j}\sigma_{k,\ell}]-[\sigma_{k,\ell}]\\
		&=\sigma_{i,j}[\sigma_{k,\ell}]+[\sigma_{i,j}]-\sigma_{k,\ell}[\sigma_{i,j}]-[\sigma_{k,\ell}]\\
	\end{align*}
	Finally, to compute $\partial_{2}^{P}(\tilde{e}_{i,k,j})$ recall that in $S_{n}$, $\sigma_{i,k}=\sigma_{i,j}\sigma_{j,k}\sigma_{i,j}^{-1}$ and therefore $\sigma_{i,k}\sigma_{i,j}=\sigma_{i,j}\sigma_{j,k}$. Hence we have:
	\begin{align*}
		\partial_{2}^{P}([\sigma_{i,j}\mid\sigma_{j,k}]-[\sigma_{i,k}\mid \sigma_{i,j}])&=\sigma_{i,j}[\sigma_{j,k}]-[\sigma_{i,j}\sigma_{j,k}]+[\sigma_{i,j}]-\sigma_{i,k}[\sigma_{i,j}]+[\sigma_{i,k}\sigma_{i,j}]-[\sigma_{i,k}]\\
		&=\sigma_{i,j}[\sigma_{j,k}]-[\sigma_{i,k}\sigma_{j,k}]+[\sigma_{i,j}]-\sigma_{i,k}[\sigma_{i,j}]+[\sigma_{i,j}\sigma_{j,k}]-[\sigma_{i,k}]\\
		&=\sigma_{i,j}[\sigma_{j,k}]+[\sigma_{i,j}]-\sigma_{i,k}[\sigma_{i,j}]-[\sigma_{i,k}]
	\end{align*}
\end{proof}
By Theorem \ref{thm: chain map}, we can use the construction of a 2-cocycle by the normalized bar resolution to define a 2-cocycle by the truncated resolution for the universal cover of a $\mathcal{K}(S_{n},1)$-space. 
\begin{theorem}\label{thm: definition of cohom class}
	Let $K$ be any $S_{n}$ module and let $E$ be an extension of $S_{n}$ by $K$. Suppose $\kappa\in\hom(\bar{P}_{2},K)$ is a representative for the cohomology class in $H^{2}(S_{n};K)$ corresponding to $E$ determined by the normalized bar resolution. Define $\kappa'\in\hom(R_{2},K)$ by:
	\begin{align*}
		\kappa'(\tilde{c}_{i,j})&=s(\sigma_{i,j})(\sigma_{i,j})\\
		\kappa'(\tilde{d}_{i,j,k,\ell})&=s(\sigma_{i,j})s(\sigma_{k,\ell})s(\sigma_{i,j}\sigma_{k,\ell})^{-1}-s(\sigma_{k,\ell})s(\sigma_{i,j})s(\sigma_{k,\ell}\sigma_{i,j})^{-1}\\
		\kappa'(\tilde{e}_{i,k,j})&=s(\sigma_{i,j})s(\sigma_{j,k})s(\sigma_{i,j}\sigma_{j,k})^{-1}-s(\sigma_{i,k})s(\sigma_{i,j})s(\sigma_{i,k}\sigma_{i,j})^{-1}
	\end{align*}
	Then $\kappa'$ is the 2-cocylce determined by the resolution corresponding to $\tilde{X}$ such that $[\kappa']$ and $[\kappa]$ represent the same group extension.
\end{theorem}
\begin{proof}
	Recall that $\bar{\mathcal{P}}$ is the normalized bar resolution and $\mathcal{R}$ is the resolution corresponding to $\tilde{X}$. Evaluating $\kappa\circ\gamma$ on the generators of $R_{2}$ we get:
	\begin{align*}
		\kappa([\sigma_{i,j}\mid\sigma_{i,j}])&=s(\sigma_{i,j})s(\sigma_{i,j})\\
		\kappa([\sigma_{i,j}\mid\sigma_{k,\ell}]-[\sigma_{k,\ell}\mid\sigma_{i,j}])&=s(\sigma_{i,j})s(\sigma_{k,\ell})s(\sigma_{i,j}\sigma_{k,\ell})^{-1}-s(\sigma_{k,\ell})s(\sigma_{i,j})s(\sigma_{k,\ell}\sigma_{i,j})^{-1}\\
		\kappa([\sigma_{i,j}\mid \sigma_{k,\ell}]-[\sigma_{i,k}\mid \sigma_{i,j} ])&=s(\sigma_{i,j})s(\sigma_{j,k})s(\sigma_{i,j}\sigma_{j,k})^{-1}-s(\sigma_{i,k})s(\sigma_{i,j})s(\sigma_{i,k}\sigma_{i,j})^{-1}
	\end{align*}
	Therefore $\kappa'=\kappa\circ\gamma_{2}$ and by Section \ref{group cohomology}, both $[\kappa']$ and $[\kappa]$ correspond to the same extension.
\end{proof}
	
\section{Cohomology class of $S_{n}$ with coefficients in $\mathbb{Z}^{n\choose2}$}\label{section on Gn}
\paragraph{Extension of $S_{n}$ by $\mathbb{Z}^{n\choose 2}$}
Let $K_{n}$ be the commutator subgroup of $PB_{n}$. Since $H_{1}(PB_{n};\mathbb{Z})\approx\mathbb{Z}^{n\choose 2}$ is the abelianization of $PB_{n}$, we have $\mathbb{Z}^{n\choose2}\approx H_{1}(PB_{n};\mathbb{Z})\approx PB_{n}/K_{n}$. Furthermore, $K_{n}$ is a characteristic subgroup of $PB_{n}$ and $PB_{n}$ is normal in $B_{n}$, therefore $K_{n}$ is normal in $B_{n}$. By the third isomorphism theorem,
	$$(B_{n}/K_{n})/(PB_{n}/K_{n})\approx B_{n}/PB_{n}\approx S_{n}$$
Let $G_{n}=B_{n}/K_{n}$, then $G_{n}$ is an extension of $S_{n}$ by $H_{1}(PB_{n};\mathbb{Z})$. Therefore we have the group extension:
	$$\begin{tikzcd}
		0\arrow{r}&\mathbb{Z}^{n\choose2}\arrow{r}{\iota_{1}}&G_{n}\arrow{r}{\pi_{1}}&S_{n}\arrow{r}&1
	\end{tikzcd}$$
where the action of $S_{n}$ on $\mathbb{Z}^{n\choose2}$ is induced by the conjugation of $B_{n}$ on $PB_{n}$. Notice the induced action of $S_{n}$ on the abelianization of the pure braid group permutes the strands of pure braids. So $S_{n}$ permutes the generators of $\mathbb{Z}^{n\choose2}$ by acting on the indices with the standard action of $S_{n}$ on unordered pairs of integers.
\paragraph{Normalized section}
Since $G_{n}$ is a quotient of $B_{n}$, let $\tilde{\sigma}_{i,j}$ represent the projection of the positively oriented half twist between the $i^{th}$ and $j^{th}$ strands, $b_{i,j}$ from \eqref{braid group}, in $G_{n}$. In particular, we need to fix a choice of normal form for each element in $S_{n}$ and choose an algorithm which takes any element of $S_{n}$, and produces the chosen normal form in terms of the generators $\sigma_{i,j}$. 
	
Note, since we consider multiplication in $G_{n}$ from left to right, we will also consider composition of permutations in $S_{n}$ from left to right for consistency. Let $p\in S_{n}$ such that $p(n)=k_{n}$, then $p\cdot \sigma_{k_{n},n}\in S_{n-1}$. Suppose $p\cdot\sigma_{k_{n},n}(n-1)=k_{n-1}$, then $p\cdot\sigma_{k_{n},n}\cdot\sigma_{k_{n-1},n-1}\in S_{n-2}$. Inductively:
$$p\cdot\sigma_{k_{n},n}\cdot\sigma_{k_{n-1},n-1}\cdots\sigma_{1,k_{1}}=1$$
Therefore $p=\sigma_{1,k_{1}}\cdot\sigma_{2,k_{2}}\cdots\sigma_{k_{n},n}$. Define the normal section $s:S_{n}\to G_{n}$ by:
$$s(p)=\tilde{\sigma}_{1,k_{1}}\tilde{\sigma}_{2,k_{2}}\cdots\tilde{\sigma}_{k_{n},n}$$
\subsection{Presentation of $G_{n}$}
\begin{table}
	\begin{center}
		\caption{Relations of $G_{n}$}\label{Table: Relations of G}
		\vspace{1em}
		\begin{tabular}{lcl}			
			
			\vspace{.5em}
			R1:&$[g_{i,j},g_{k,\ell}]=1$& for all $i,j,k,\ell$\\
			
			\vspace{.5em}
			R2:& $\tilde{\sigma}_{i,j}^{2}=g_{i,j}$ &for all $i$\\
			
			\vspace{1em}
			R3:& 
			$\tilde{\sigma}_{i,j}\tilde{\sigma}_{k,j}\tilde{\sigma}_{i,j}^{-1}=\tilde{\sigma}_{i,k}$
			&if\hspace{1em}$k<i<j;\hspace{.5em}i<j<k;\hspace{.5em}\text{or}\hspace{.5em}j<k<i$\\
				
			\vspace{.5em}
			R4:& $\tilde{\sigma}_{i,j}^{-1}\tilde{\sigma}_{j,k}\tilde{\sigma}_{i,j}=\tilde{\sigma}_{i,k}$ & 
			if\hspace{1em}$i<k<j;\hspace{.5em}j<i<k;\hspace{.5em}\text{or}\hspace{.5em}k<j<i$\\
				
			\vspace{.5em}
			R5:&$[\tilde{\sigma}_{i,j},\tilde{\sigma}_{k,\ell}]=
			\begin{cases}
				g_{i,k}g_{i,\ell}^{-1}g_{j,k}^{-1}g_{j,\ell}\\
				g_{k,i}^{-1}g_{k,j}g_{i,\ell}g_{\ell,j}^{-1}\\
				1
			\end{cases}$ & 
			$\begin{aligned}
					i<k<j<\ell\\
					k<i<\ell<j\\
					\text{otherwise}
				\end{aligned}$\\
				
			\vspace{.5em}
			R6:& $\tilde{\sigma}_{i,j}g_{k,\ell}\tilde{\sigma}_{i,j}^{-1}=
			g_{\sigma_{i,j}(k),\sigma_{i,j}(\ell)}$&for all $i,j,k,\ell\in\{1,\ldots,n\}$\\
		\end{tabular}\\
	\end{center}
\end{table}
\paragraph{Generators of $G_{n}$}
Let $g_{i,j}$ $(1\leq i<j\leq n)$ be the commuting generators of $\mathbb{Z}^{n\choose2}$. Since $\mathbb{Z}^{n\choose2}\approx PB_{n}/K_{n}$, each $g_{i,j}$ represents the projections of a pure braid, $b_{i,j}^{2}$, in $G_{n}$. In particular $g_{i,j}$ is the positively oriented full twist of the $i^{\text{th}}$ and $j^{\text{th}}$ strands while $g_{i,j}^{-1}$ represents the negatively oriented full twist between the $i^{\text{th}}$ and $j^{\text{th}}$ strands. By the choice of section, each $\sigma_{i,j}$ ($\sigma_{i,j}^{-1}$) in the generating set of $S_{n}$ lifts to the positively (negatively) oriented half twist between the $i^{\text{th}}$ and $j^{\text{th}}$ strands, denoted $\tilde{\sigma}_{i,j}$ ($\tilde{\sigma}_{i,j}^{-1}$). Thus by section \ref{section on building presentations}, a generating set for $G_{n}$ is:
	$$\{g_{i,j}\}\cup \{\tilde{\sigma}_{i,j}\}$$
where $1\leq i<j\leq n$.
\subsubsection*{Relations of $G_{n}$}
Note that we will remove the assumption $i<j$ for many of the statements and proofs in this section for more efficient notation. Recall from section \ref{section on building presentations}, to construct a presentation for $G_{n}$, we need to include relations of $\mathbb{Z}^{n\choose2}$, lift relations of $S_{n}$, and conjugate generators of $\mathbb{Z}^{n\choose2}$ by generators of $S_{n}$. Therefore a full list of all the relations in $G_{n}$ are given in Table \ref{Table: Relations of G}. Notice R1 is the included relation of $\mathbb{Z}^{n\choose2}$, R2-R5 are the lifted relations of $S_{n}$, and R6 is the conjugation of relations of generators. 

The rest of this section will prove the relations given in Table \ref{Table: Relations of G}. Since $\mathbb{Z}^{n\choose 2}$ embeds into $G_{n}$, the relation $[g_{i,j},g_{k,\ell}]$ is preserved in $G_{n}$. Therefore R1 holds in $G_{n}$ as the inclusion of the relation in $\mathbb{Z}^{n\choose2}$. We begin with a theorem describing how to determine elements of $\mathbb{Z}^{n\choose 2}$ in $G_{n}$.
\begin{theorem}\label{thm: strand diagrams}
	Let $k\in\iota_{1}(\mathbb{Z}^{n\choose 2})$. Then $k$ is determined by the winding numbers of a strand diagram.
\end{theorem}
\begin{proof}
	Let $k\in\iota_{1}(\mathbb{Z}^{n\choose2})$, then $k$ is represented by a pure braid since $\mathbb{Z}^{n\choose2}\approx PB_{n}/K_{n}$. Since $[g_{i,j},g_{k,\ell}]=1$, $k$ can be written uniquely as:
		$$k=\prod_{i=1}^{n-1}\prod_{j=i+1}^{n}g_{i,j}^{\epsilon_{i,j}}$$
	where $\epsilon_{i,j}\in\mathbb{Z}$ is the number of positively oriented full twists between the $i$ and $j$ strands. By definition $\epsilon_{i,j}$ is the winding number between the $i$th and $j$th strands. Therefore the winding numbers of all combinations of strands determine $k$. 
\end{proof}
Since $\tilde{\sigma}_{i,j}$ represents the positively oriented half twist between the $i^{\text{th}}$ and $j^{\text{th}}$ strands, R2 is a relation of $G_{n}$. We begin by proving relations R3 and R4.
\begin{lemma}\label{R3}
	If $k<i<j$, $i<j<k$, or $j<k<i$, then 
		$$\tilde{\sigma}_{i,j}\tilde{\sigma}_{j,k}\tilde{\sigma}_{i,j}^{-1}=\tilde{\sigma}_{i,k}$$
\end{lemma}
\begin{proof}
	Note that $b_{i,j}b_{j,k}b_{i,j}^{-1}=b_{i,k}$ if $i<j<k$, $k<i<j$, or $j<k<i$ is a relation of $B_{n}$ in \eqref{braid group}. Since $G_{n}$ is a quotient of $B_{n}$ with $\tilde{\sigma}_{i,j}$ representing the projection of $b_{i,j}$ in $G_{n}$, $\tilde{\sigma}_{i,j}\tilde{\sigma}_{j,k}\tilde{\sigma}_{i,j}^{-1}=\tilde{\sigma}_{i,k}$ is a relation of $G_{n}$.
\end{proof}
\begin{lemma}\label{R4}
	If $i<k<j$, $j<i<k$, or $k<j<i$, then
	$$\tilde{\sigma}_{i,j}\tilde{\sigma}_{j,k}\tilde{\sigma}_{i,j}^{-1}\tilde{\sigma}_{i,k}^{-1}=g_{i,j}g_{k,j}^{-1}$$
\end{lemma}
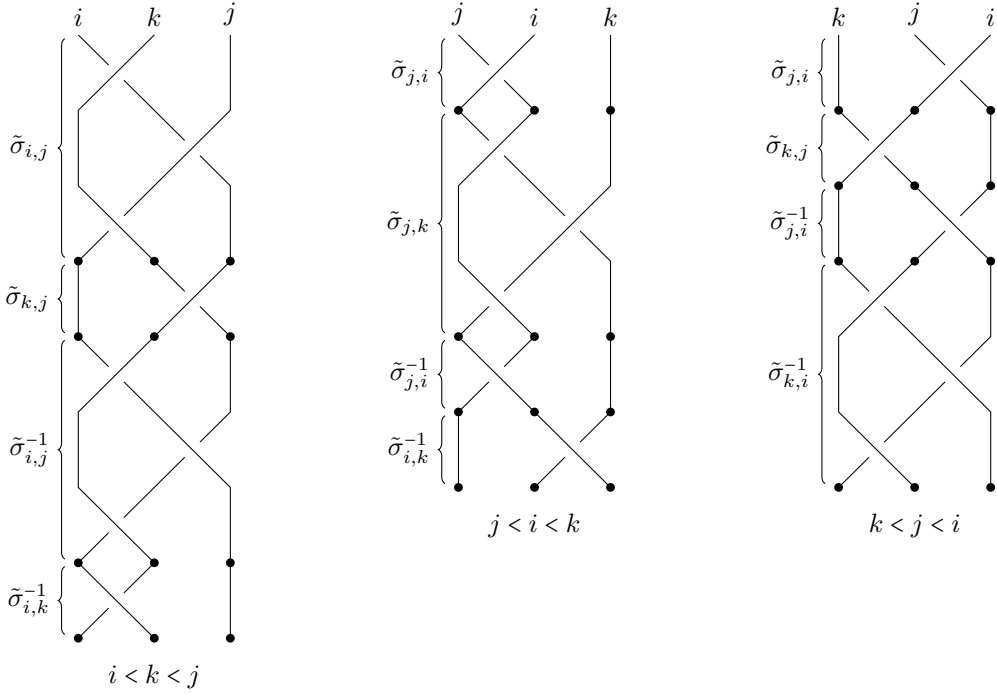
\begin{figure}
	\caption{$\tilde{\sigma}_{i,j}\tilde{\sigma}_{j,k}\tilde{\sigma}_{i,j}^{-1}\tilde{\sigma}_{i,k}^{-1}$} \label{R4 diagram}
	$$\begin{tikzpicture}
		\node[above] at (0,0){$i$};
		\node[above] at (1,0){$k$};
		\node[above] at (2,0){$j$};
		\draw (0,0) to (.4,-.4);
		\draw (.6,-.6) to (1.4,-1.4);
		\draw (1.6,-1.6) to (2,-2) to (2,-3);
		\draw (1,0) to (0,-1) to (0,-2) to (1,-3);
		\draw (2,0) to (2,-1) to (.6,-2.4);
		\draw (.4,-2.6) to (0,-3);
		\node[circle, fill, scale=.35] at (0,-3){};
		\node[circle, fill, scale=.35] at (1,-3){};
		\node[circle, fill, scale=.35] at (2,-3){};
		\draw[decoration={brace,mirror,raise=5pt},decorate]
		(0,-.05) -- node[left=7pt] {$\tilde{\sigma}_{i,j}$} (0,-2.95);
		\draw (0,-3) to (0,-4);
		\draw (1,-3) to (1.4,-3.4);
		\draw (1.6,-3.6) to (2,-4);
		\draw (2,-3) to (1,-4);
		\node[circle, fill, scale=.35] at (0,-4){};
		\node[circle, fill, scale=.35] at (1,-4){};
		\node[circle, fill, scale=.35] at (2,-4){};
		\draw[decoration={brace,mirror,raise=5pt},decorate]
		(0,-3.05) -- node[left=7pt] {$\tilde{\sigma}_{k,j}$} (0,-3.95);
		\draw (0,-4) to (.4,-4.4);
		\draw (.6,-4.6) to (2,-6) to (2,-7);
		\draw (1,-4) to (0,-5) to (0,-6) to (1,-7);
		\draw (2,-4) to (2,-5) to (1.6,-5.4);
		\draw (1.4,-5.6) to (.6,-6.4);
		\draw (.4,-6.6) to (0,-7);
		\node[circle, fill, scale=.35] at (0,-7){};
		\node[circle, fill, scale=.35] at (1,-7){};
		\node[circle, fill, scale=.35] at (2,-7){};
		\draw[decoration={brace,mirror,raise=5pt},decorate]
		(0,-4.05) -- node[left=7pt] {$\tilde{\sigma}_{i,j}^{-1}$} (0,-6.95);
		\draw (0,-7) to (1,-8);
		\draw (1,-7) to (.6,-7.4);
		\draw (.4,-7.6) to (0,-8);
		\draw (2,-7) to (2,-8);
		\node[circle, fill, scale=.35] at (0,-8){};
		\node[circle, fill, scale=.35] at (1,-8){};
		\node[circle, fill, scale=.35] at (2,-8){};
		\draw[decoration={brace,mirror,raise=5pt},decorate]
		(0,-7.05) -- node[left=7pt] {$\tilde{\sigma}_{i,k}^{-1}$} (0,-7.95);
		\node at (1,-8.5){$i<k<j$};
		\node[above] at (5,0){$j$};
		\node[above] at (6,0){$i$};
		\node[above] at (7,0){$k$};
		\draw (7,0) to (7,-1);
		\draw (5,0) to (5.4,-.4);
		\draw (5.6,-.6) to (6,-1);
		\draw (6,0) to (5,-1);
		\node[circle, fill, scale=.35] at (7,-1){};
		\node[circle, fill, scale=.35] at (6,-1){};
		\node[circle, fill, scale=.35] at (5,-1){};
		\draw[decoration={brace,mirror,raise=5pt},decorate]
		(5,-.05) -- node[left=7pt] {$\tilde{\sigma}_{j,i}$} (5,-.95);
		\draw (5,-1) to (5.4,-1.4);
		\draw (5.6,-1.6) to (6.4,-2.4);
		\draw (6.6,-2.6) to (7,-3) to (7,-4);
		\draw (6,-1) to (5,-2) to (5,-3) to (6,-4);
		\draw (7,-1) to (7,-2) to (5.6,-3.4);
		\draw (5.4,-3.6) to (5,-4);
		\node[circle, fill, scale=.35] at (7,-4){};
		\node[circle, fill, scale=.35] at (6,-4){};
		\node[circle, fill, scale=.35] at (5,-4){};
		\draw[decoration={brace,mirror,raise=5pt},decorate]
		(5,-1.05) -- node[left=7pt] {$\tilde{\sigma}_{j,k}$} (5,-3.95);
		\draw (5,-4) to (6,-5);
		\draw (6,-4) to (5.6,-4.4);
		\draw (5.4,-4.6) to (5,-5);
		\draw (7,-4) to (7,-5);
		\node[circle, fill, scale=.35] at (7,-5){};
		\node[circle, fill, scale=.35] at (6,-5){};
		\node[circle, fill, scale=.35] at (5,-5){};
		\draw[decoration={brace,mirror,raise=5pt},decorate]
		(5,-4.05) -- node[left=7pt] {$\tilde{\sigma}_{j,i}^{-1}$} (5,-4.95);
		\draw (5,-5) to (5,-6);
		\draw (6,-5) to (7,-6);
		\draw (7,-5) to (6.6,-5.4);
		\draw (6.4,-5.6) to (6,-6);
		\node[circle, fill, scale=.35] at (7,-6){};
		\node[circle, fill, scale=.35] at (6,-6){};
		\node[circle, fill, scale=.35] at (5,-6){};
		\draw[decoration={brace,mirror,raise=5pt},decorate]
		(5,-5.05) -- node[left=7pt] {$\tilde{\sigma}_{i,k}^{-1}$} (5,-5.95);
		\node at (6,-6.5){$j<i<k$};
		\node[above] at (10,0){$k$};
		\node[above] at ( 11,0){$j$};
		\node[above] at (12,0){$i$};
		\draw (10,0) to (10,-1);
		\draw (11,0) to (11.4,-.4);
		\draw (11.6,-.6) to (12,-1);
		\draw (12,0) to (11,-1);
		\node[circle, fill, scale=.35] at (10,-1){};
		\node[circle, fill, scale=.35] at (11,-1){};
		\node[circle, fill, scale=.35] at (12,-1){};
		\draw[decoration={brace,mirror,raise=5pt},decorate]
		(10,-.05) -- node[left=7pt] {$\tilde{\sigma}_{j,i}$} (10,-.95);
		\draw (10,-1) to (10.4,-1.4);
		\draw (10.6,-1.6) to (11,-2);
		\draw (11,-1) to (10,-2);
		\draw (12,-1) to (12,-2);
		\node[circle, fill, scale=.35] at (10,-2){};
		\node[circle, fill, scale=.35] at (11,-2){};
		\node[circle, fill, scale=.35] at (12,-2){};
		\draw[decoration={brace,mirror,raise=5pt},decorate]
		(10,-1.05) -- node[left=7pt] {$\tilde{\sigma}_{k,j}$} (10,-1.95);
		\draw (11,-2) to (12,-3);
		\draw (12,-2) to (11.6,-2.4);
		\draw (11.4,-2.6) to (11,-3);
		\draw (10,-2) to (10,-3);
		\node[circle, fill, scale=.35] at (10,-3){};
		\node[circle, fill, scale=.35] at (11,-3){};
		\node[circle, fill, scale=.35] at (12,-3){};
		\draw[decoration={brace,mirror,raise=5pt},decorate]
		(10,-2.05) -- node[left=7pt] {$\tilde{\sigma}_{j,i}^{-1}$} (10,-2.95);
		\draw (11,-3) to (10,-4) to (10,-5) to (11,-6);
		\draw (10,-3) to (10.4,-3.4);
		\draw (10.6,-3.6) to (12,-5) to (12,-6);
		\draw (12,-3) to (12,-4) to (11.6,-4.4);
		\draw (11.4,-4.6) to (10.6,-5.4);
		\draw (10.4,-5.6) to (10,-6);
		\node[circle, fill, scale=.35] at (10,-6){};
		\node[circle, fill, scale=.35] at (11,-6){};
		\node[circle, fill, scale=.35] at (12,-6){};
		\draw[decoration={brace,mirror,raise=5pt},decorate]
		(10,-3.05) -- node[left=7pt] {$\tilde{\sigma}_{k,i}^{-1}$} (10,-5.95);
		\node at (11,-6.5){$k<j<i$};
	\end{tikzpicture}$$
\end{figure}
\begin{proof}
	By rewriting the relations as $\sigma_{i,j}\sigma_{j,k}\sigma_{i,j}^{-1}\sigma_{i,k}^{-1}=1$ in $S_{n}$, we get $\tilde{\sigma}_{i,j}\tilde{\sigma}_{j,k}\tilde{\sigma}_{i,j}^{-1}\tilde{\sigma}_{i,k}^{-1}$ is in $\ker\pi$. Therefore $\tilde{\sigma}_{i,j}\tilde{\sigma}_{j,k}\tilde{\sigma}_{i,j}^{-1}\tilde{\sigma}_{i,k}^{-1}\in\mathbb{Z}^{n\choose2}$ and by Theorem \ref{thm: strand diagrams} is determined by Figure \ref{R4 diagram}.
	Suppose $i<k<j$, notice that by Figure \ref{R4 diagram} strands $i$ and $k$ have both a clockwise and counterclockwise twist, therefore the winding number of $g_{i,k}$ is zero. Furthermore, strand $j$ passes over strand $i$ then under, so $g_{i,j}$ has a winding number of $1$. Also, strand $j$ passes under strand $k$ then over it, thus $g_{k,j}$ has a winding number of $-1$.
		
	If $j<i<k$, $g_{j,i}$ has a winding number of $1$ in Figure \ref{R4 diagram}. Now $k$ passes over $i$ twice and $g_{i,k}$ has a winding number of zero. Furthermore $k$ passes under then over strand $j$, so $g_{j,k}$ has a winding number of $-1$.
		
	Suppose $k<j<i$, by Figure \ref{R4 diagram} strand $k$ passes under $i$, over $j$, then under both $i$ and $j$. Therefore $g_{k,i}$ has a winding number of zero while $g_{k,j}$ has a winding number of $-1$. Now, $i$ passes over then under $j$, so $g_{j,i}$ has winding number $1$ since $i>j$.
	\end{proof}
\paragraph{Relation 3} Notice that relation R3 is proven by Lemma \ref{R3}. To prove relation R4, we need to show Lemma \ref{R4} implies $\tilde{\sigma}_{i,j}^{-1}\tilde{\sigma}_{k,\ell}\tilde{\sigma}_{i,j}=\tilde{\sigma}_{i,k}$. To prove this we need to prove the case of R6 when $|\{i,j\}\cap\{k,\ell\}|=1$.
\begin{theorem}\label{thm: conj full by half with intersection}
	Suppose $|\{i,j\}\cap\{k,\ell\}|=1$, then:
	$$\tilde{\sigma}_{i,j}g_{k,\ell}\tilde{\sigma}_{i,j}^{-1}=g_{\sigma_{i,j}(k),\sigma_{i,j}(\ell)}$$
\end{theorem}
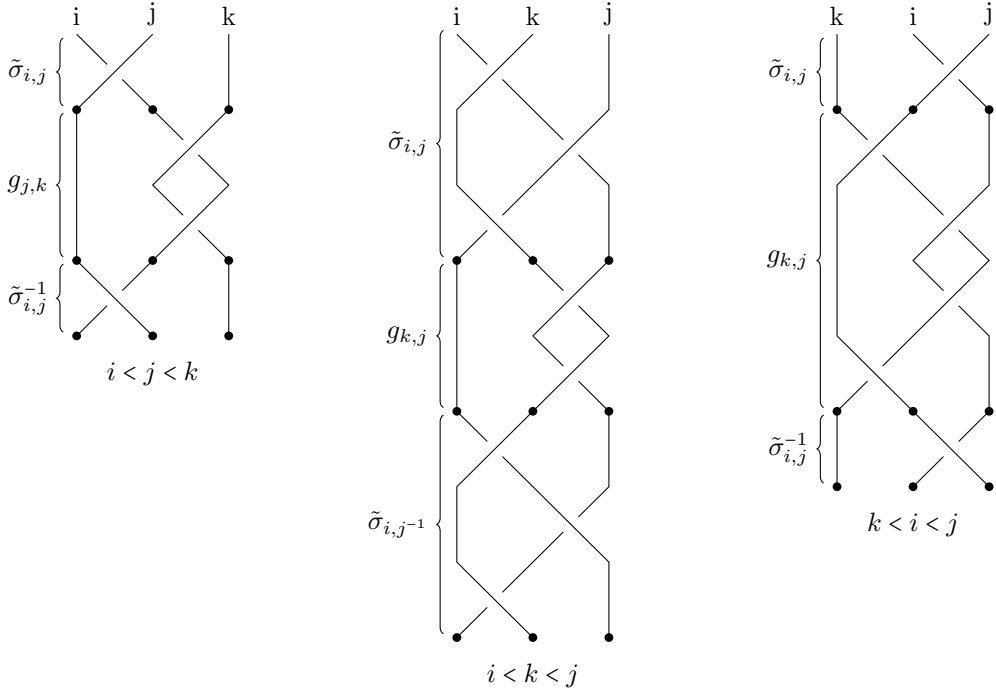
\begin{figure}
	\caption{$\tilde{\sigma}_{i,j}g_{j,k}\tilde{\sigma}_{i,j}^{-1}$}\label{conj full twist with intersection j}
	$$\begin{tikzpicture}
		\node[above] at (0,0){i};
		\node[above] at (1,0){j};
		\node[above] at (2,0){k};
		\draw (1,0) to (0,-1) to (0,-3) to (1,-4);
		\draw (0,0) to (.4,-.4);
		\draw (.6,-.6) to (1.4,-1.4);
		\draw (2,0) to (2,-1) to (1,-2) to (1.4,-2.4);
		\node[circle,fill, scale=.35] at (0,-1){};
		\node[circle,fill, scale=.35] at (1,-1){};
		\node[circle,fill, scale=.35] at (2,-1){};  
		\draw[decoration={brace,mirror,raise=5pt},decorate]
		(0,-.05) -- node[left=7pt] {$\tilde{\sigma}_{i,j}$} (0,-.95);
		\draw (1.6,-1.6) to (2,-2) to (.6,-3.4);
		\draw (1.6,-2.6) to (2,-3) to (2,-4);
		\draw (.4,-3.6) to (0,-4);
		\node[circle,fill, scale=.35] at (0,-3){};
		\node[circle,fill, scale=.35] at (1,-3){};
		\node[circle,fill, scale=.35] at (2,-3){};  
		\draw[decoration={brace,mirror,raise=5pt},decorate]
		(0,-1.05) -- node[left=7pt] {$g_{j,k}$} (0,-2.95);
		\node[circle,fill, scale=.35] at (0,-4){};
		\node[circle,fill, scale=.35] at (1,-4){};
		\node[circle,fill, scale=.35] at (2,-4){};  
		\draw[decoration={brace,mirror,raise=5pt},decorate]
		(0,-3.05) -- node[left=7pt] {$\tilde{\sigma}_{i,j}^{-1}$} (0,-3.95);
		\node[above] at (5,0){i};
		\node[above] at (6,0){k};
		\node[above] at (7,0){j};
		\draw (5,0) to (5.4,-.4);
		\draw (5.6,-.6) to (6.4,-1.4);
		\draw (6.6,-1.6) to (7,-2) to (7,-3);
		\draw (6,0) to (5,-1) to (5,-2) to (6,-3);
		\draw (7,0) to (7,-1) to (5.6,-2.4);
		\draw (5.4,-2.6) to (5,-3);
		\node[circle,fill, scale=.35] at (5,-3){};
		\node[circle,fill, scale=.35] at (6,-3){};
		\node[circle,fill, scale=.35] at (7,-3){};  
		\draw[decoration={brace,mirror,raise=5pt},decorate]
		(5,.05) -- node[left=7pt] {$\tilde{\sigma}_{i,j}$} (5,-2.95);
		\draw (7,-3) to (6,-4) to (6.4,-4.4);
		\draw (6.6,-4.6) to (7,-5);
		\draw (6,-3) to (6.4,-3.4);
		\draw (6.6,-3.6) to (7,-4) to (6,-5);
		\draw (5,-3) to (5,-5);
		\node[circle,fill, scale=.35] at (5,-5){};
		\node[circle,fill, scale=.35] at (6,-5){};
		\node[circle,fill, scale=.35] at (7,-5){};  
		\draw[decoration={brace,mirror,raise=5pt},decorate]
		(5,-3.05) -- node[left=7pt] {$g_{k,j}$} (5,-4.95);
		\draw (5,-5) to (5.4,-5.4);
		\draw (5.6,-5.6) to (7,-7) to (7,-8);
		\draw (6,-5) to (5,-6) to (5,-7) to (6,-8);
		\draw (7,-5) to (7,-6) to (6.6,-6.4);
		\draw (6.4,-6.6) to (5.6,-7.4);
		\draw (5.4,-7.6) to (5,-8);
		\node[circle,fill, scale=.35] at (5,-8){};
		\node[circle,fill, scale=.35] at (6,-8){};
		\node[circle,fill, scale=.35] at (7,-8){};  
		\draw[decoration={brace,mirror,raise=5pt},decorate]
		(5,-5.05) -- node[left=7pt] {$\tilde{\sigma}_{i,j^{-1}}$} (5,-7.95);
		\node[above] at (10,0){k};
		\node[above] at (11,0){i};
		\node[above] at (12,0){j};
		\draw (10,0) to (10,-1);
		\draw (11,0) to (11.4,-.4);
		\draw (11.6,-.6) to (12,-1);
		\draw (12,0) to (11,-1);
		\node[circle,fill, scale=.35] at (10,-1){};
		\node[circle,fill, scale=.35] at (11,-1){};
		\node[circle,fill, scale=.35] at (12,-1){};  
		\draw[decoration={brace,mirror,raise=5pt},decorate]
		(10,-.05) -- node[left=7pt] {$\tilde{\sigma}_{i,j}$} (10,-.95);
		\draw (11,-1) to (10,-2) to (10,-4) to (11,-5);
		\draw (10,-1) to (10.4,-1.4);
		\draw (10.6,-1.6) to (11.4,-2.4);
		\draw (11.6,-2.6) to (12,-3) to (10.6,-4.4);
		\draw (10.4,-4.6) to (10,-5);
		\draw (12,-1) to (12,-2) to (11,-3) to (11.4,-3.4);
		\draw (11.6,-3.6) to (12,-4) to (12,-5);
		\node[circle,fill, scale=.35] at (10,-5){};
		\node[circle,fill, scale=.35] at (11,-5){};
		\node[circle,fill, scale=.35] at (12,-5){};  
		\draw[decoration={brace,mirror,raise=5pt},decorate]
		(10,-1.05) -- node[left=7pt] {$g_{k,j}$} (10,-4.95);
		\draw (10,-5) to (10,-6);
		\draw (11,-5) to (12,-6);
		\draw (12,-5) to (11.6,-5.4);
		\draw (11.4,-5.6) to (11,-6);
		\node[circle,fill, scale=.35] at (10,-6){};
		\node[circle,fill, scale=.35] at (11,-6){};
		\node[circle,fill, scale=.35] at (12,-6){};  
		\draw[decoration={brace,mirror,raise=5pt},decorate]
		(10,-5.05) -- node[left=7pt] {$\tilde{\sigma}_{i,j}^{-1}$} (10,-5.95);
		\node at (1,-4.5){$i<j<k$};
		\node at (6,-8.5){$i<k<j$};
		\node at (11,-6.5){$k<i<j$};
	\end{tikzpicture}$$
\end{figure}
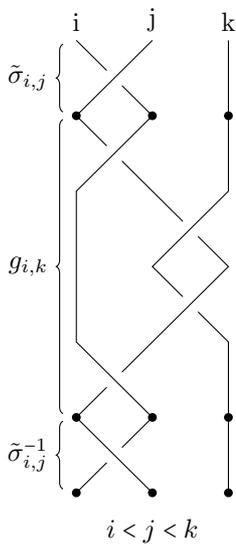
\begin{figure}
	\caption{$\tilde{\sigma}_{i,j}g_{i,k}\tilde{\sigma}_{i,j}^{-1}$}\label{conj full twist intersection i}
	$$\begin{tikzpicture}
		\node[above] at (0,0){i};
		\node[above] at (1,0){j};
		\node[above] at (2,0){k};
		\draw (0,0) to (.4,-.4);
		\draw (.6,-.6) to (1,-1);
		\draw (1,0) to (0,-1);
		\draw (2,0) to (2,-1);
		\node[circle,fill, scale=.35] at (0,-1){};
		\node[circle,fill, scale=.35] at (1,-1){};
		\node[circle,fill, scale=.35] at (2,-1){};  
		\draw[decoration={brace,mirror,raise=5pt},decorate]
		(0,-.05) -- node[left=7pt] {$\tilde{\sigma}_{i,j}$} (0,-.95);
		\draw (1,-1) to (0,-2) to (0,-4) to (1,-5);
		\draw (0,-1) to (.4,-1.4);
		\draw (.6,-1.6) to (1.4,-2.4);
		\draw (1.6,-2.6) to (2,-3) to (.6,-4.4);
		\draw (.4,-4.6) to (0,-5);
		\draw (2,-1) to (2,-2) to (1,-3) to (1.4,-3.4);
		\draw (1.6,-3.6) to (2,-4) to (2,-5);
		\node[circle,fill, scale=.35] at (0,-5){};
		\node[circle,fill, scale=.35] at (1,-5){};
		\node[circle,fill, scale=.35] at (2,-5){};  
		\draw[decoration={brace,mirror,raise=5pt},decorate]
		(0,-1.05) -- node[left=7pt] {$g_{i,k}$} (0,-4.95);
		\draw (2,-5) to (2,-6);
		\draw (1,-5) to (.6,-5.4);
		\draw (.4,-5.6) to (0,-6);
		\draw (0,-5) to (1,-6);
		\node[circle,fill, scale=.35] at (0,-6){};
		\node[circle,fill, scale=.35] at (1,-6){};
		\node[circle,fill, scale=.35] at (2,-6){};  
		\draw[decoration={brace,mirror,raise=5pt},decorate]
		(0,-5.05) -- node[left=7pt] {$\tilde{\sigma}_{i,j}^{-1}$} (0,-5.95);
		\node at (1,-6.5){$i<j<k$};
		\node[above] at (5,0){i};
		\node[above] at (6,0){k};
		\node[above] at (7,0){j};
		\draw (6,0) to (5,-1) to (5,-2) to (6,-3);
		\draw (7,0) to (7,-1) to (5.6,-2.4);
		\draw (5.4,-2.6) to (5,-3);
		\draw (5,0) to (5.4,-.4);
		\draw (5.6,-.6) to (6.4,-1.4);
		\draw (6.6,-1.6) to (7,-2) to (7,-3);
		\node[circle,fill, scale=.35] at (5,-3){};
		\node[circle,fill, scale=.35] at (6,-3){};
		\node[circle,fill, scale=.35] at (7,-3){};  
		\draw[decoration={brace,mirror,raise=5pt},decorate]
		(5,-.05) -- node[left=7pt] {$\tilde{\sigma}_{i,j}$} (5,-2.95);
		\draw (6,-3) to (5,-4) to (5.4,-4.4);
		\draw (5.6,-4.6) to (6,-5);
		\draw (5,-3) to (5.4,-3.4);
		\draw (5.6,-3.6) to (6,-4) to (5,-5);
		\draw (7,-3) to (7,-5);
		\node[circle,fill, scale=.35] at (5,-5){};
		\node[circle,fill, scale=.35] at (6,-5){};
		\node[circle,fill, scale=.35] at (7,-5){};  
		\draw[decoration={brace,mirror,raise=5pt},decorate]
		(5,-3.05) -- node[left=7pt] {$g_{i,k}$} (5,-4.95);
		\draw (6,-5) to (5,-6) to (5,-7) to (6,-8);
		\draw (5,-5) to (5.4,-5.4);
		\draw (5.6,-5.6) to (7,-7) to (7,-8);
		\draw (7,-5) to (7,-6) to (6.6,-6.4);
		\draw (6.4,-6.6) to (5.6,-7.4);
		\draw (5.4,-7.6) to (5,-8);
		\node[circle,fill, scale=.35] at (5,-8){};
		\node[circle,fill, scale=.35] at (6,-8){};
		\node[circle,fill, scale=.35] at (7,-8){};  
		\draw[decoration={brace,mirror,raise=5pt},decorate]
		(5,-5.05) -- node[left=7pt] {$\tilde{\sigma}_{i,j}^{-1}$} (5,-7.95);
		\node at (6,-8.5){$i<k<j$};
		\node[above] at (10,0){k};
		\node[above] at (11,0){i};
		\node[above] at (12,0){j};
		\draw (10,0) to (10,-1);
		\draw (11,0) to (11.4,-.4);
		\draw (11.6,-.6) to (12,-1);
		\draw (12,0) to (11,-1);
		\node[circle,fill, scale=.35] at (10,-1){};
		\node[circle,fill, scale=.35] at (11,-1){};
		\node[circle,fill, scale=.35] at (12,-1){};  
		\draw[decoration={brace,mirror,raise=5pt},decorate]
		(10,-.05) -- node[left=7pt] {$\tilde{\sigma}_{i,j}$} (10,-.95);
		\draw (10,-1) to (10.4,-1.4);
		\draw (10.6,-1.6) to (11,-2) to (10,-3);
		\draw (11,-1) to (10,-2) to (10.4,-2.4);
		\draw (10.6,-2.6) to (11,-3);
		\draw (12,-1) to (12,-3);
		\node[circle,fill, scale=.35] at (10,-3){};
		\node[circle,fill, scale=.35] at (11,-3){};
		\node[circle,fill, scale=.35] at (12,-3){};  
		\draw[decoration={brace,mirror,raise=5pt},decorate]
		(10,-1.05) -- node[left=7pt] {$g_{k,i}$} (10,-2.95);
		\draw (10,-3) to (10,-4);
		\draw (11,-3) to (12,-4);
		\draw (12,-3) to (11.6,-3.4);
		\draw (11.4,-3.6) to (11,-4);
		\node[circle,fill, scale=.35] at (10,-4){};
		\node[circle,fill, scale=.35] at (11,-4){};
		\node[circle,fill, scale=.35] at (12,-4){};  
		\draw[decoration={brace,mirror,raise=5pt},decorate]
		(10,-3.05) -- node[left=7pt] {$\tilde{\sigma}_{i,j}^{-1}$} (10,-3.95);
		\node at (11,-4.5){$k<i<j$};
	\end{tikzpicture}$$
\end{figure}
\begin{proof}
	Since $\mathbb{Z}^{n\choose2}$ is normal in $G_{n}$, $\tilde{\sigma}_{i,j}g_{j,k}\tilde{\sigma}_{i,j}^{-1}\in \mathbb{Z}^{n\choose2}$. Suppose $\max\{i,j\}\in\{k,\ell\}$, without loss of generality assume $j=k$. Then there are three cases: $i<j<k$, $i<k<j$ and $k<i<j$.
		
	By Figure \ref{conj full twist with intersection j}, for both cases $i<j<k$ and $k<i<j$, the winding number of strand $j$ associated to both $i$ and $k$ is zero. Furthermore, in both cases there is a full clockwise twist between the $i^{\text{th}}$ and $j^{\text{th}}$ strands, so the resulting twist for both cases are $g_{i,k}=g_{\sigma_{i,j}(j),\sigma_{i,j}(k)}$. 
		
	If $i<k<j$, notice that the $j^{\text{th}}$ strand passes under the $k^{\text{th}}$ strand twice and over the $i^{\text{th}}$ strand twice. Therefore there is no twist between the $j^{\text{th}}$ strand and either of the other two. Furthermore, the $k^{\text{th}}$ strand passes over the $i^{\text{th}}$ strand, then under it, and over strand $i$ twice more. Therefore there is one full clockwise twist between strands $i$ and $k$. Thus $\tilde{\sigma}_{i,j}\tilde{\sigma}_{j,k}\tilde{\sigma}_{i,j}^{-1}=g_{i,k}$.
		
	Now, suppose $\min\{i,j\}\in\{k,\ell\}$ and without loss of generality assume $i=\ell$. There are the three cases: $i<j<k$, $i<k<j$, and $k<i<j$ as shown in Figure \ref{conj full twist intersection i}. In both cases $i<k<j$ and $k<i<j$ the $i^{\text{th}}$ strand only passes under all other strands, therefore the winding number of any twist involving the $i^{\text{th}}$ strand is zero. If $k<i<j$ then Figure \ref{conj full twist intersection i} shows a clockwise full twist between strands $k$ and $j$. If $i<k<j$, notice the $j^{\text{th}}$ strand passes under the $k^{\text{th}}$ strand twice, then over it, then under again. Therefore strands $k$ and $j$ have the equivalent of a full clockwise twist and the winding number between  the $k^{\text{th}}$ and $j^{\text{th}}$ strands is $1$.
		
	If $i<j<k$, Figure \ref{conj full twist intersection i} shows the $i^{\text{th}}$ and $j^{\text{th}}$ strands have both a full clockwise and counterclockwise twist; therefore the winding number for $g_{i,j}$ is zero. Furthermore, Figure \ref{conj full twist intersection i} shows strands $k$ and $j$ have a positive full twist while strands $i$ and $k$ have a winding number of zero. Thus:
		$$\sigma_{i,j}g_{k,\ell}\sigma_{i,j}=g_{\sigma_{i,j}(k),\sigma_{i,j}(\ell)}$$
	if $|\{i,j\}\cap\{k,\ell\}|=1$.
\end{proof}
\begin{theorem}\label{relation 4}
	Relation R4 holds in $G_{n}$.
\end{theorem}
\begin{proof}
	By Theorem \ref{thm: conj full by half with intersection}, we have $\tilde{\sigma}_{i,j}\tilde{\sigma}_{j,k}\tilde{\sigma}_{i,j}^{-1}\tilde{\sigma}_{i,k}^{-1}=g_{i,j}g_{k,j}^{-1}$. Multiplying on the right by $\tilde{\sigma}_{i,k}$ and on the left by $g_{i,j}^{-1}$, by relation R1 we have
	\begin{align*}
		\tilde{\sigma}_{i,j}\tilde{\sigma}_{j,k}\tilde{\sigma}_{i,j}^{-1}\tilde{\sigma}_{i,k}^{-1}&=g_{i,j}g_{k,j}^{-1}\\			g_{i,j}^{-1}\tilde{\sigma}_{i,j}\tilde{\sigma}_{j,k}\tilde{\sigma}_{i,j}^{-1}&=g_{k,j}^{-1}\sigma_{i,k}
	\end{align*}
	By applying R2 and Theorem \ref{thm: conj full by half with intersection} we have:
	\begin{align*}
		g_{i,j}^{-1}\tilde{\sigma}_{i,j}\tilde{\sigma}_{j,k}\tilde{\sigma}_{i,j}^{-1}&=g_{k,j}^{-1}\sigma_{i,k}\\
		\tilde{\sigma}_{i,j}^{-2}\tilde{\sigma}_{i,j}\tilde{\sigma}_{k,j}\tilde{\sigma}_{i,j}^{-1}&=\tilde{\sigma}_{i,k}g_{i,j}^{-1}\\
		\tilde{\sigma}_{i,j}^{-1}\tilde{\sigma}_{k,j}\tilde{\sigma}_{i,j}^{-1}g_{i,j}&=\tilde{\sigma}_{i,k}
	\end{align*}
	By R2: $g_{i,j}=\tilde{\sigma}_{i,j}^{2}$. Therefore the theorem and R4 are proven.
\end{proof}
\begin{theorem}\label{thm: commutator of half twists}
	Suppose $|\{i,j\}\cap\{k,\ell\}|\neq 1$. Then
		$$[\tilde{\sigma}_{i,j},\tilde{\sigma}_{k,\ell}]=
		\begin{cases}
			g_{i,k}g_{i,\ell}^{-1}g_{j,k}^{-1}g_{j,\ell}\hspace{1em}&i<k<j<\ell\\
			g_{k,i}^{-1}g_{k,j}g_{i,\ell}g_{\ell,j}^{-1}\hspace{1em}&k<i<\ell<j\\
			1\hspace{1em}&\text{otherwise}
		\end{cases}$$
\end{theorem}
\begin{figure}
	\caption{$[\tilde{\sigma}_{i,j},\tilde{\sigma}_{k,\ell}]$}\label{braided commutator}
	$$\begin{tikzpicture}
		\node[above] at (0,0){$i$};
		\node[above] at (1,0){$k$};
		\node[above] at (2,0){$j$};
		\node[above] at (3,0){$\ell$};
		\draw (0,0) to (.4,-.4);
		\draw (.6,-.6) to (1.4,-1.4);
		\draw (1.6,-1.6) to (2,-2) to (2,-3);
		\draw (1,0) to (0,-1) to (0,-2) to (1,-3);
		\draw (2,0) to (2,-1) to (.6,-2.4);
		\draw (.4,-2.6) to (0,-3);
		\node[circle, fill, scale=.35] at (0,-3){};
		\node[circle, fill, scale=.35] at (1,-3){};
		\node[circle, fill, scale=.35] at (2,-3){};
		\node[circle, fill, scale=.35] at (3,-3){};
		\draw (3,0) to (3,-3);
		\draw[decoration={brace,mirror,raise=5pt},decorate]
		(0,-.05) -- node[left=7pt] {$\tilde{\sigma}_{i,j}$} (0,-2.95);
		\draw (0,-3) to (0,-6);
		\draw (2,-3) to (1,-4) to (1,-5) to (2,-6);
		\draw (1,-3) to (1.4,-3.4);
		\draw (1.6,-3.6) to (2.4,-4.4);
		\draw (2.6,-4.6) to (3,-5) to (3,-6);
		\draw (3,-3) to (3,-4) to (1.6,-5.4);
		\draw (1.4,-5.6) to (1,-6);
		\node[circle, fill, scale=.35] at (0,-6){};
		\node[circle, fill, scale=.35] at (1,-6){};
		\node[circle, fill, scale=.35] at (2,-6){};
		\node[circle, fill, scale=.35] at (3,-6){};
		\draw (3,0) to (3,-3);
		\draw[decoration={brace,mirror,raise=5pt},decorate]
		(0,-3.05) -- node[left=7pt] {$\tilde{\sigma}_{k,\ell}$} (0,-5.95);
		\draw (0,-6) to (.4,-6.4);
		\draw (.6,-6.6) to (2,-8) to (2,-9);
		\draw (1,-6) to (0,-7) to (0,-8) to (1,-9);
		\draw (2,-6) to (2,-7) to (1.6,-7.4);
		\draw (1.4,-7.6) to (.6,-8.4);
		\draw (.4,-8.6) to (0,-9);
		\draw (3,-6) to (3,-9);
		\node[circle, fill, scale=.35] at (0,-9){};
		\node[circle, fill, scale=.35] at (1,-9){};
		\node[circle, fill, scale=.35] at (2,-9){};
		\node[circle, fill, scale=.35] at (3,-9){};
		\draw (3,0) to (3,-3);
		\draw[decoration={brace,mirror,raise=5pt},decorate]
		(0,-6.05) -- node[left=7pt] {$\tilde{\sigma}_{i,j}^{-1}$} (0,-8.95);
		\draw (0,-9) to (0,-12);
		\draw (1,-9) to (1.4,-9.4);
		\draw (1.6,-9.6) to (3,-11) to (3,-12);
		\draw (2,-9) to (1,-10) to (1,-11) to (2,-12);
		\draw (3,-9) to (3,-10) to (2.6,-10.4);
		\draw (2.4,-10.6) to (1.6,-11.4);
		\draw (1.4,-11.6) to (1,-12);
		\node[circle, fill, scale=.35] at (0,-12){};
		\node[circle, fill, scale=.35] at (1,-12){};
		\node[circle, fill, scale=.35] at (2,-12){};
		\node[circle, fill, scale=.35] at (3,-12){};
		\draw (3,0) to (3,-3);
		\draw[decoration={brace,mirror,raise=5pt},decorate]
		(0,-9.05) -- node[left=7pt] {$\tilde{\sigma}_{k,\ell}^{-1}$} (0,-11.95);
	\end{tikzpicture}$$
\end{figure}
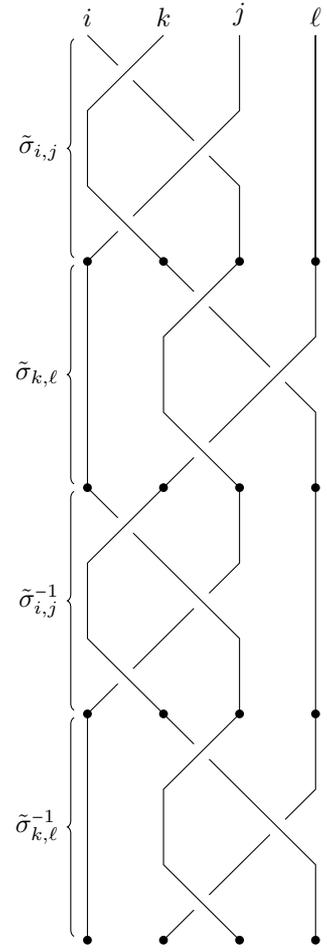
\begin{figure}
	\caption{$\tilde{\sigma}_{i,j}g_{k,\ell}\tilde{\sigma}_{i,j}^{-1}$}\label{conjugation of full by half, braiden no intersection}
	$$\begin{tikzpicture}
		\node[above] at (0,0){$i$};
		\node[above] at (1,0){$k$};
		\node[above] at (2,0){$j$};
		\node[above] at (3,0){$\ell$};
		\draw (0,0) to (.4,-.4);
		\draw (.6,-.6) to (1.4,-1.4);
		\draw (1.6,-1.6) to (2,-2) to (2,-3);
		\draw (1,0) to (0,-1) to (0,-2) to (1,-3);
		\draw (2,0) to (2,-1) to (.6,-2.4);
		\draw (.4,-2.6) to (0,-3);
		\draw (3,0) to (3,-3);
		\node[circle, fill, scale=.35] at (0,-3){};
		\node[circle, fill, scale=.35] at (1,-3){};
		\node[circle, fill, scale=.35] at (2,-3){};
		\node[circle, fill, scale=.35] at (3,-3){};
		\draw (3,0) to (3,-3);
		\draw[decoration={brace,mirror,raise=5pt},decorate]
		(0,-.05) -- node[left=7pt] {$\tilde{\sigma}_{i,j}$} (0,-2.95);
		\draw (1,-3) to (1.4,-3.4);
		\draw (1.6,-3.6) to (2.4,-4.4);
		\draw (2.6,-4.6) to (3,-5) to (1.6,-6.4);
		\draw (1.4,-6.6) to (1,-7);
		\draw (2,-3) to (1,-4) to (1,-6) to (2,-7);
		\draw (3,-3) to (3,-4) to (2,-5) to (2.4,-5.4);
		\draw (2.6,-5.6) to (3,-6) to (3,-7);
		\draw (0,-3) to (0,-7);
		\node[circle, fill, scale=.35] at (0,-7){};
		\node[circle, fill, scale=.35] at (1,-7){};
		\node[circle, fill, scale=.35] at (2,-7){};
		\node[circle, fill, scale=.35] at (3,-7){};
		\draw[decoration={brace,mirror,raise=5pt},decorate]
		(0,-3.05) -- node[left=7pt] {$g_{k,\ell}$} (0,-6.95);
		\draw (1,-7) to (0,-8) to (0,-9) to (1,-10);
		\draw (0,-7) to (.4,-7.4);
		\draw (.6,-7.6) to (2,-9) to (2,-10);
		\draw (2,-7) to (2,-8) to (1.6,-8.4);
		\draw (1.4,-8.6) to (.6,-9.4);
		\draw (.4,-9.6) to (0,-10);
		\draw (3,-7) to (3,-10);
		\node[circle, fill, scale=.35] at (0,-10){};
		\node[circle, fill, scale=.35] at (1,-10){};
		\node[circle, fill, scale=.35] at (2,-10){};
		\node[circle, fill, scale=.35] at (3,-10){};
		\draw[decoration={brace,mirror,raise=5pt},decorate]
		(0,-7.05) -- node[left=7pt] {$\tilde{\sigma}_{i,j}^{-1}$} (0,-9.95);
		
		\node[above] at (6,0){$k$};
		\node[above] at (7,0){$i$};
		\node[above] at (8,0){$\ell$};
		\node[above] at (9,0){$j$};
		\draw (7,0) to (7.4,-.4);
		\draw (7.6,-.6) to (8.4,-1.4);
		\draw (8.6,-1.6) to (9,-2) to (9,-3);
		\draw (6,0) to (6,-3);
		\draw (9,0) to (9,-1) to (7.6,-2.4);
		\draw (7.4,-2.6) to (7,-3);
		\draw (8,0) to (7,-1) to (7,-2) to (8,-3);
		\node[circle, fill, scale=.35] at (6,-3){};
		\node[circle, fill, scale=.35] at (7,-3){};
		\node[circle, fill, scale=.35] at (8,-3){};
		\node[circle, fill, scale=.35] at (9,-3){};
		\draw[decoration={brace,mirror,raise=5pt},decorate]
		(6,-.05) -- node[left=7pt] {$\tilde{\sigma}_{i,j}$} (6,-2.95);
		\draw (7,-3) to (6,-4) to (6,-6) to (7,-7);
		\draw (6,-3) to (6.4,-3.4);
		\draw (6.6,-3.6) to (7.4,-4.4);
		\draw (7.6,-4.6) to (8,-5) to (6.6,-6.4);
		\draw (6.4,-6.6) to (6,-7);
		\draw (8,-3) to (8,-4) to (7,-5) to (7.4,-5.4);
		\draw (7.6,-5.6) to (8,-6) to (8,-7);
		\draw (9,-3) to (9,-7);
		\node[circle, fill, scale=.35] at (6,-7){};
		\node[circle, fill, scale=.35] at (7,-7){};
		\node[circle, fill, scale=.35] at (8,-7){};
		\node[circle, fill, scale=.35] at (9,-7){};
		\draw[decoration={brace,mirror,raise=5pt},decorate]
		(6,-3.05) -- node[left=7pt] {$g_{k,\ell}$} (6,-6.95);
		\draw (6,-7) to (6,-10);
		\draw (7,-7) to (7.4,-7.4);
		\draw (7.6,-7.6) to (9,-9) to (9,-10);
		\draw (8,-7) to (7,-8) to (7,-9) to (8,-10);
		\draw (9,-7) to (9,-8) to (8.6,-8.4);
		\draw (8.4,-8.6) to (7.6,-9.4);
		\draw (7.4,-9.6) to (7,-10);
		\node[circle, fill, scale=.35] at (6,-10){};
		\node[circle, fill, scale=.35] at (7,-10){};
		\node[circle, fill, scale=.35] at (8,-10){};
		\node[circle, fill, scale=.35] at (9,-10){};
		\draw[decoration={brace,mirror,raise=5pt},decorate]
		(6,-7.05) -- node[left=7pt] {$\tilde{\sigma}_{i,j}^{-1}$} (6,-9.95);
	\end{tikzpicture}$$
\end{figure}
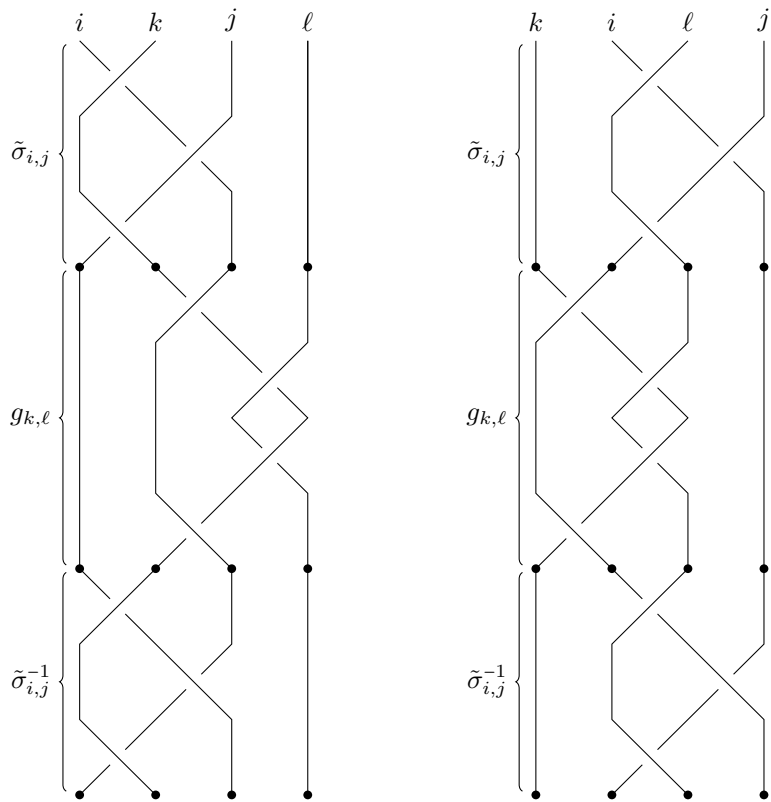
\begin{proof}
	Suppose $\{i,j\}\cap\{k,\ell\}=\emptyset$. Without loss of generality, assume $i<j$ and $k<\ell$. If $j<k$ or $\ell<i$ then $[\tilde{\sigma}_{i,j},\tilde{\sigma}_{k,\ell}]=1$ since $G_{n}$ is a quotient of $B_{n}$. Since $[\tilde{\sigma}_{i,j},\tilde{\sigma}_{k,\ell}]=[\tilde{\sigma}_{k,\ell},\tilde{\sigma}_{i,j}]^{-1}$, it suffices to show the cases $i<k<j<\ell$ and $i<k<\ell<j$.
		
	If $i<k<\ell<j$, then $(i-k)(i-\ell)(j-k)(j-\ell)>0$. Therefore $[b_{i,j},b_{k,\ell}]=1$ in \eqref{braid group}. Since $G_{n}$ is a quotient of $B_{n}$, $[\tilde{\sigma}_{i,j},\tilde{\sigma}_{k,\ell}]=1$. Therefore it remains to show the case $i<k<j<\ell$.
		
	Now, suppose $i<k<j<\ell$. By Figure \ref{braided commutator} the $i^{\text{th}}$ strand passes under the $k^{\text{th}}$ strand then over, therefore $g_{i,k}$ has a winding number of $1$. Furthermore strand $i$ passes under the $j^{\text{th}}$ strand twice; also the $i^{\text{th}}$ strand passes over the $\ell^{\text{th}}$ strand once then under it and therefore has a counterclockwise full twist. Therefore $g_{i,j}$ has a winding number of zero and $g_{i,\ell}$ has a winding number $-1$. 
	Now, the $j^{\text{th}}$ strand passes under the strand $k$ then over it, therefore $g_{k,j}$ has a winding number of $-1$. Strand $j$ passes under strand $\ell$ then under it and thus $g_{j,\ell}$ has a winding number of $1$. Finally strand $\ell$ passes over strand $k$ twice, therefore $g_{k,\ell}$ has a winding number of zero. Thus the theorem is proven.
\end{proof}
	
To finish the proof this presentation for $G_{n}$ is correct, it remains to show the second part of R5.
	
\begin{theorem}
	If $|\{i,j\}\cap\{k,\ell\}|\neq 1$, then
		$$\tilde{\sigma}_{i,j}g_{k,\ell}\tilde{\sigma}_{i,j}^{-1}=g_{k,\ell}$$
\end{theorem}
	
\begin{proof}
	First suppose $\{i,j\}\cap\{k,\ell\}=\{i,j\}$. Then by relation R3 we have
	\begin{align*}
		\tilde{\sigma}_{i,j}g_{k,\ell}\tilde{\sigma}_{i,j}^{-1}&=\tilde{\sigma}_{i,j}\tilde{\sigma}_{i,j}^{2}\tilde{\sigma}_{i,j}^{-1}\\
		&=\tilde{\sigma}_{i,j}^{2}\\
		&=g_{i,j}
	\end{align*}
	Now, suppose $\{i,j\}\cap\{k,\ell\}=\emptyset$. If $i<k<j<\ell$, then by Figure \ref{conjugation of full by half, braiden no intersection}, strands $i$ and $k$ have one full clockwise twist followed by a full counterclockwise twist, therefore the winding number corresponding to $g_{i,k}$ is zero. Also, the $\ell^{\text{th}}$ strand has one full clockwise twist with the $k^{\text{th}}$ strand and does not interact with either of the other strands. Therefore $g_{i\ell}$ and $g_{j,\ell}$ have winding numbers zero while $g_{k,\ell}$ has a winding number of $1$. Furthermore, strand $j$ passes under strand $k$ twice, so $g_{k,j}$ has a corresponding winding number of zero as well. Thus, if $i<k<j<\ell$ then the theorem holds.
		
	Now, suppose $k<i<\ell<j$. By Figure \ref{conjugation of full by half, braiden no intersection} the $i^{\text{th}}$ strand always passes under the other strands, therefore $g_{k,i}$, $g_{i,\ell}$, and $g_{i,j}$ all have winding number zero. Furthermore, the $j^{\text{th}}$ strand passes under the strand $\ell$ twice and over strand $k$ twice, therefore $g_{k,j}$ and $g_{\ell, j}$ also have winding numbers zero. Finally, Figure \ref{conjugation of full by half, braiden no intersection} shows one full clockwise twist between the $k^{\text{th}}$ and $\ell^{\text{th}}$ strands. Therefore $\tilde{\sigma}_{i,j}g_{k,\ell}\tilde{\sigma}_{i,j}^{-1}=g_{k,\ell}$ if $k<i<\ell<j$.
		
	Otherwise $(i-k)(i-\ell)(j-k)(j-\ell)>0$, therefore by R3 and R5 we have:
	\begin{align*}
		\tilde{\sigma}_{i,j}g_{k,\ell}\tilde{\sigma}_{i,j}^{-1}&=\tilde{\sigma}_{i,j}\tilde{\sigma}_{k,\ell}^{2}\tilde{\sigma}_{i,j}^{-1}\\
		&=\tilde{\sigma}_{i,j}\tilde{\sigma}_{k,\ell}\tilde{\sigma}_{i,j}^{-1}\tilde{\sigma}_{i,j}\tilde{\sigma}_{k,\ell}\tilde{\sigma}_{i,j}^{-1}\\
		&=\tilde{\sigma}_{i,j}\tilde{\sigma}_{k,\ell}\tilde{\sigma}_{i,j}^{-1}\tilde{\sigma}_{k,\ell}^{-1}\tilde{\sigma}_{k,\ell}\tilde{\sigma}_{i,j}\tilde{\sigma}_{k,\ell}\tilde{\sigma}_{i,j}^{-1}\\
		&=\tilde{\sigma}_{i,j}\tilde{\sigma}_{k,\ell}\tilde{\sigma}_{i,j}^{-1}\tilde{\sigma}_{k,\ell}^{-1}\tilde{\sigma}_{k,\ell}\tilde{\sigma}_{i,j}\tilde{\sigma}_{k,\ell}\tilde{\sigma}_{i,j}^{-1}\tilde{\sigma}_{k,\ell}^{-1}\tilde{\sigma}_{k,\ell}\\
		&=1\cdot\tilde{\sigma}_{k,\ell}\cdot 1\cdot\tilde{\sigma}_{k,\ell}\\
		&=\tilde{\sigma}_{k,\ell}^{2}\\	
		&=g_{k,\ell}
	\end{align*}
\end{proof}
\subsection{Cohomology class in $H^{2}(S_{n};\mathbb{Z}^{n\choose2})$}
Given the group presentation for $G_{n}$ with relations in Table \ref{Table: Relations of G}, we can determine the cohomology class representing the group extension of $S_{n}$ by $\mathbb{Z}^{n\choose2}$ by Theorem \ref{thm: definition of cohom class}. Recall the algorithm described at the beginning of Section \ref{section on Gn} to determine the chosen normal form of a permutation $p\in S_{n}$ does the following:
\begin{itemize}
	\item Rewrites $p$ as a word in the generators of $S_{n}$ given in \eqref{symmetric group}.
	\item For all $1\leq i<j\leq n$, replaces $\sigma_{i,j}^{2}$ with $1_{S_{n}}$.
	\item If $\{i,j\}\cap\{k,\ell\}=\emptyset$ and $\max\{k,\ell\}<\max\{i,j\}$, replaces $\sigma_{i,j}\sigma_{k,\ell}$ with $\sigma_{k,\ell}\sigma_{i,j}$.
	\item If $j=\max\{i,j,k\}$, replaces $\sigma_{i,j}\sigma_{j,k}$ with $\sigma_{i,k}\sigma_{i,j}$.
	\item If $i=\max\{i,j,k\}$, replaces $\sigma_{i,j}\sigma_{j,k}$ with $\sigma_{j,k}\sigma_{i,k}$.
\end{itemize}
The normal section $s:S_{n}\to G_{n}$ defined on generators of $S_{n}$ by $s(\sigma_{i,j})=\tilde{\sigma}_{i,j}$ determines the lift of permutations in $S_{n}$ by the chosen normal form. Therefore Theorem \ref{thm: definition of cohom class} implies we can compute the image for a cocycle representing the corresponding cohomology class by the relations of $G_{n}$. In particular, we can use the relations in Table \ref{Table: Relations of G}. Since we lift any element of $S_{n}$ by the choice of normal form, we first determine the normal forms for products of two transpositions. 
\begin{lemma}\label{normal form, no intersection}
	Suppose $\{i,j\}\cap\{k,\ell\}=\emptyset$, then 
		$$s(\sigma_{k,\ell}\sigma_{i,j})=s(\sigma_{i,j}\sigma_{k,\ell})=\begin{cases}
			\tilde{\sigma}_{i,j}\tilde{\sigma}_{k,\ell}\hspace{1em}&j<\ell\\
			\tilde{\sigma}_{k,\ell}\tilde{\sigma}_{i,j}\hspace{1em}&\ell<j
		\end{cases}$$
\end{lemma}
\begin{proof}
	Both $\sigma_{k,\ell}\sigma_{i,j}$ and $\sigma_{i,j}\sigma_{k,\ell}$ have the same normal form since they both describe the same permutation. If $\ell<j$ then we take
	$\sigma_{i,j}\sigma_{k,\ell}\sigma_{i,j}\sigma_{k,\ell}=1$ and $\sigma_{i,j}\sigma_{k,\ell}$ becomes $\sigma_{k,\ell}\sigma_{i,j}$. Therefore the normal form of $\sigma_{i,j}\sigma_{k,\ell}$ is the same as the normal form for $\sigma_{k,\ell}\sigma_{i,j}$ and $s(\sigma_{i,j}\sigma_{k,\ell})=\tilde{\sigma}_{k,\ell}\tilde{\sigma}_{i,j}$. Now, suppose $j<\ell$. Then we take the transposition fixing $k$ first, yielding $\sigma_{i,j}\sigma_{k,\ell}\sigma_{k,\ell}\sigma_{i,j}$ and thus $s(\sigma_{i,j}\sigma_{k,\ell})=\tilde{\sigma}_{k,\ell}\tilde{\sigma}_{i,j}$.
\end{proof}
\begin{lemma}\label{lemma: normal form for intersection} 
	The normal form for products of transpositions with intersection are:
		$$s(\sigma_{i,k}\sigma_{i,j})=s(\sigma_{i,j}\sigma_{j,k})=\begin{cases}
			\tilde{\sigma}_{i,k}\tilde{\sigma}_{i,j}\hspace{1em}&\max\{i,k,j\}=j\\
			\tilde{\sigma}_{i,j}\tilde{\sigma}_{j,k}\hspace{1em}&\max\{i,k,j\}=k\\
			\tilde{\sigma}_{k,j}\tilde{\sigma}_{k,i}\hspace{1em}&\max\{i,k,j\}=i
		\end{cases}$$
\end{lemma}
The proof of Lemma \ref{lemma: normal form for intersection} is a computation similar to the proof of Lemma \ref{normal form, no intersection} and omitted.
Now that we have the normal forms for products of generators, it is possible to compute a representative for the cohomology class describing $G_{n}$ as an extension of $S_{n}$ by $\mathbb{Z}^{n\choose 2}$.
	
\begin{theorem}\label{thm: cohom of G_{n}}
	Let:
	$$\begin{tikzcd}
			0\arrow{r}&\mathbb{Z}^{n\choose2}\arrow{r}{\iota_{1}}&G_{n}\arrow{r}{\pi_{1}}&S_{n}\arrow{r}&1
		\end{tikzcd}$$
	be the group extension where the action of $S_{n}$ on $\mathbb{Z}^{n\choose2}$ is determined by the conjugation of pure braids by half twists in $B_{n}$. The cocycle, $\phi$, defined by:
	\begin{align*}
		\phi(\tilde{c}_{i,j})&=g_{i,j}\\
		\phi(\tilde{d}_{i,j,k,\ell})&=\begin{cases}
			g_{i,k}-g_{i,\ell}-g_{k,j}+g_{j,\ell}\hspace{1em}&i<k<j<\ell\\
			-g_{i,k}+g_{k,j}+g_{i,\ell}-g_{\ell,j}\hspace{1em}&k<i<\ell<j\\
			0\hspace{1em}&\text{otherwise}
		\end{cases}\\
		\phi(\tilde{e}_{i,k,j})&=\begin{cases}
			g_{i,j}-g_{k,j}\hspace{1em}&i<k<j,\hspace{.5em}j<i<k,\hspace{.5em}k<j<i\\
			0\hspace{1em}&\text{otherwise}
		\end{cases}
	\end{align*}
	is a representative for the cohomology class of $H^{2}(S_{n};\mathbb{Z}^{n\choose2})$ corresponding to this extension.
\end{theorem}
Before we prove Theorem \ref{thm: cohom of G_{n}}, note that the image of our cocycle is the abelian group $\mathbb{Z}^{n\choose2}$ with additive notation. However the computations to determine elements of $\mathbb{Z}^{n\choose 2}$ is done within the extension $G_{n}$ using multiplicative notation.
\begin{proof}
	First, by Theorem \ref{thm: definition of cohom class} we have
	\begin{align*}
		\phi(\tilde{c}_{i,j})&=s(\sigma_{i,j})s(\sigma_{i,j})\\
		&=\tilde{\sigma}_{i,j}\tilde{\sigma}_{i,j}\\
		&=g_{i,j}
	\end{align*}
	Now, consider $\tilde{d}_{i,j,k,\ell}$, without loss of generality assume $i<j$ and $k<\ell$. By Lemma 4.8 $s(\sigma_{k,\ell}\sigma_{i,j})=s(\sigma_{i,j}\sigma_{k,\ell})$ depends on $\max\{j,\ell\}$. Suppose $j<\ell$, then by Lemma \ref{normal form, no intersection}, $s(\sigma_{k,\ell}\sigma_{i,j})=\tilde{\sigma}_{i,j}\tilde{\sigma}_{k,\ell}$. Therefore by Theorem \ref{thm: definition of cohom class} and relation R5
	\begin{align*}
		\phi(\tilde{d}_{i,j,k,\ell})&=\tilde{\sigma}_{i,j}\tilde{\sigma}_{k,\ell}(\tilde{\sigma}_{i,j}\tilde{\sigma}_{k,\ell})^{-1}-\tilde{\sigma}_{k,\ell}\tilde{\sigma}_{i,j}(\tilde{\sigma}_{i,j}\tilde{\sigma}_{k,\ell})^{-1}\\
		&=-\tilde{\sigma}_{k,\ell}\tilde{\sigma}_{i,j}\tilde{\sigma}_{k,\ell}^{-1}\tilde{\sigma}_{i,j}^{-1}\\
		&=-\begin{cases}
			g_{i,k}^{-1}g_{i,\ell}g_{k,j}g_{j,\ell}^{-1}\hspace{1em}&i<k<j<\ell\\
				0\hspace{1em}&\text{otherwise}
		\end{cases}\\
		&=-\begin{cases}
			-g_{i,k}+g_{i,\ell}+g_{k,j}-g_{j,\ell}\hspace{1em}&i<k<j<\ell\\
			0\hspace{1em}&\text{otherwise}
		\end{cases}\\
		&=\begin{cases}
			g_{i,k}-g_{i,\ell}-g_{k,j}+g_{j,\ell}\hspace{1em}&i<k<j<\ell\\
			0\hspace{1em}&\text{otherwise}
		\end{cases}
	\end{align*}
	Thus for $\phi(\tilde{d}_{i,j,k,\ell})$ it remains to show the case $\ell<j$. By Lemma \ref{normal form, no intersection}, Theorem \ref{thm: definition of cohom class}, and R5
	\begin{align*}
		\phi(\tilde{d}_{i,j,k,\ell})&=\tilde{\sigma}_{i,j}\tilde{\sigma}_{k,\ell}(\tilde{\sigma}_{k,\ell}\tilde{\sigma}_{i,j})^{-1}-\tilde{\sigma}_{k,\ell}\tilde{\sigma}_{i,j}(\tilde{\sigma}_{k,\ell}\tilde{\sigma}_{i,j})^{-1}\\
		&=\tilde{\sigma}_{i,j}\tilde{\sigma}_{k,\ell}\tilde{\sigma}_{i,j}^{-1}\tilde{\sigma}_{k,\ell}^{-1}\\
		&=\begin{cases}
			g_{k,i}^{-1}g_{k,j}g_{i,\ell}g_{\ell,j}^{-1}\hspace{1em}&k<i<\ell<j\\
			0\hspace{1em}&\text{otherwise}
		\end{cases}\\
		&=\begin{cases}
			-g_{k,i}+g_{k,j}+g_{i,\ell}-g_{\ell,j}\hspace{1em}&k<i<\ell<j\\
			0\hspace{1em}&\text{otherwise}
		\end{cases}
	\end{align*}
	Thus it remains to compute the image of $\tilde{e}_{i,k,j}$. By Lemma \ref{lemma: normal form for intersection} there are three cases for the normal form of the permutation $\sigma_{i,k}\sigma_{i,j}=\sigma_{i,j}\sigma_{j,k}$ depending on $\max\{i,k,j\}$. Suppose $\max\{i,k,j\}=j$, then by Theorem \ref{thm: definition of cohom class} and Lemma \ref{lemma: normal form for intersection} we have:
	\begin{align*}
		\phi(\tilde{e}_{i,k,j})&=s(\sigma_{i,j})s(\sigma_{j,k})s(\sigma_{i,j}\sigma_{k,j})^{-1}-s(\sigma_{i,k})s(\sigma_{i,j})s(\sigma_{i,k}\sigma_{i,j})^{-1}\\
		&=\tilde{\sigma}_{i,j}\tilde{\sigma}_{k,j}(\tilde{\sigma}_{i,k}\tilde{\sigma}_{i,j})^{-1}-\tilde{\sigma}_{i,k}\tilde{\sigma}_{i,j}(\tilde{\sigma}_{i,k}\tilde{\sigma}_{i,j})^{-1}\\
		&=\tilde{\sigma}_{i,j}\tilde{\sigma}_{k,j}\tilde{\sigma}_{i,j}^{-1}\tilde{\sigma}_{i,k}^{-1}
	\end{align*}
	By relations R3 and R4 together with Theorem \ref{thm: conj full by half with intersection} we get:
	\begin{align*}
		\phi(\tilde{e}_{i,k,j})&=\tilde{\sigma}_{i,j}\tilde{\sigma}_{k,j}\tilde{\sigma}_{i,j}^{-1}\tilde{\sigma}_{i,k}^{-1}\\
		&=\begin{cases}
			\tilde{\sigma}_{i,k}\tilde{\sigma}_{i,k}^{-1}\hspace{1em}&k<i<j\\
			g_{i,j}\tilde{\sigma}_{i,j}^{-1}\tilde{\sigma}_{k,j}\tilde{\sigma}_{i,j}g_{i,j}^{-1}\tilde{\sigma}_{i,k}^{-1}\hspace{1em}&i<k<j
		\end{cases}\\
		&=\begin{cases}
			0\hspace{1em}&k<i<j\\
			g_{i,j}\tilde{\sigma}_{i,k}g_{i,j}^{-1}\tilde{\sigma}_{i,k}^{-1}\hspace{1em}&i<k<j
		\end{cases}\\
		&=\begin{cases}
			0\hspace{1em}&k<i<j\\
			g_{i,j}g_{k,j}^{-1}\hspace{1em}&i<k<j
		\end{cases}\\
		&=\begin{cases}
			0\hspace{1em}&k<i<j\\
			g_{i,j}-g_{k,j}\hspace{1em}&i<k<j
		\end{cases}
	\end{align*}
	Hence we are done if $\max\{i,k,j\}=j$.
		
	Now, let $\max\{i,k,j\}=k$, then by Lemma \ref{lemma: normal form for intersection} $s(\sigma_{i,k}\sigma_{i,j})=\tilde{\sigma}_{i,j}\tilde{\sigma}_{j,k}$. Therefore, as above we have
	\begin{align*}
		\phi(\tilde{e}_{i,k,j})&=\tilde{\sigma}_{i,j}\tilde{\sigma}_{j,k}(\tilde{\sigma}_{i,j}\tilde{\sigma}_{j,k})^{-1}-\tilde{\sigma}_{i,k}\tilde{\sigma}_{i,j}(\tilde{\sigma}_{i,j}\tilde{\sigma}_{j,k})^{-1}\\
		&=-\tilde{\sigma}_{i,k}\tilde{\sigma}_{i,j}\tilde{\sigma}_{j,k}^{-1}\tilde{\sigma}_{i,j}^{-1}
	\end{align*}
	Now, if $i<j<k$ then R3 applies to $\tilde{\sigma}_{i,j}\tilde{\sigma}_{j,k}^{-1}\tilde{\sigma}_{i,j}^{-1}$ and if $j<i<k$ then R4 applies:
	\begin{align*}
		\phi(\tilde{e}_{i,k,j})&=-\begin{cases}
			\tilde{\sigma}_{i,k}\tilde{\sigma}_{i,k}^{-1}\hspace{1em}&i<j<k\\
			\tilde{\sigma}_{i,k}g_{i,j}\tilde{\sigma}_{i,j}^{-1}\tilde{\sigma}_{j,k}^{-1}\tilde{\sigma}_{i,j}g_{i,j}^{-1}\hspace{1em}&j<i<k
			\end{cases}\\
		&=-\begin{cases}
			0\hspace{1em}&i<j<k\\
			\tilde{\sigma}_{i,k}g_{i,j}\tilde{\sigma}_{i,k}^{-1}g_{i,j}^{-1}\hspace{1em}&j<i<k
		\end{cases}\\
		&=-\begin{cases}
			0\hspace{1em}&i<j<k\\
			g_{i,j}^{-1}g_{k,j}\hspace{1em}&j<i<k
		\end{cases}\\
		&=\begin{cases}
			0\hspace{1em}&i<j<k\\
			g_{i,j}-g_{k,j}\hspace{1em}&j<i<k
		\end{cases}
	\end{align*}
		
	It remains to prove the result if $\max\{i,k,j\}=i$. By Lemma \ref{lemma: normal form for intersection}, $\max\{i,k,j\}=i$ implies $s(\sigma_{i,j}\sigma_{i,j})=\tilde{\sigma}_{k,j}\tilde{\sigma}_{k,i}$. Thus
	\begin{align*}
		\phi(\tilde{e}_{i,k,j})&=\tilde{\sigma}_{j,i}\tilde{\sigma}_{j,k}(\tilde{\sigma}_{j,k}\tilde{\sigma}_{k,i})^{-1}-\tilde{\sigma}_{k,i}\tilde{\sigma}_{j,i}(\tilde{\sigma}_{j,k}\tilde{\sigma}_{k,i})^{-1}\\
		&=\tilde{\sigma}_{j,i}\tilde{\sigma}_{j,k}\tilde{\sigma}_{k,i}^{-1}\tilde{\sigma}_{j,k}^{-1}-\tilde{\sigma}_{k,i}\tilde{\sigma}_{i,j}\tilde{\sigma}_{k,i}^{-1}\tilde{\sigma}_{j,k}^{-1}
	\end{align*}
	If $j<k<i$, then relation R3 applies to $\tilde{\sigma}_{j,k}\tilde{\sigma}_{k,i}^{-1}\tilde{\sigma}_{j,k}^{-1}$ while R4 applies if $k<j<i$. Therefore,
	\begin{align*}
		\tilde{\sigma}_{j,i}\tilde{\sigma}_{j,k}\tilde{\sigma}_{k,i}^{-1}\tilde{\sigma}_{j,k}^{-1}&=\begin{cases}
			\tilde{\sigma}_{j,i}\tilde{\sigma}_{j,i}^{-1}\hspace{1em}&j<k<i\\
			\tilde{\sigma}_{j,i}g_{k,j}\tilde{\sigma}_{k,j}^{-1}\tilde{\sigma}_{k,i}^{-1}\tilde{\sigma}_{k,j}g_{k,j}^{-1}\hspace{1em}&k<j<i
			\end{cases}\\
		&=\begin{cases}
			0\hspace{1em}&j<k<i\\
			\tilde{\sigma}_{j,i}g_{k,j}\tilde{\sigma}_{j,i}^{-1}g_{k,j}^{-1}\hspace{1em}&k<j<i
		\end{cases}\\
		&=\begin{cases}
			0\hspace{1em}&j<k<i\\
			g_{k,i}g_{k,j}^{-1}\hspace{1em}&k<j<i
		\end{cases}
	\end{align*}
	Furthermore, if $j<k<i$ then R3 applies to $\tilde{\sigma}_{k,i}\tilde{\sigma}_{j,i}\tilde{\sigma}_{k,i}^{-1}$ while R4 applies if $k<j<i$. Hence
	\begin{align*}
		\tilde{\sigma}_{k,i}\tilde{\sigma}_{j,i}\tilde{\sigma}_{k,i}^{-1}\tilde{\sigma}_{j,k}^{-1}
		&=\begin{cases}
			\tilde{\sigma}_{j,k}\tilde{\sigma}_{j,k}^{-1}\hspace{1em}&j<k<i\\
			g_{k,i}\tilde{\sigma}_{k,i}^{-1}\tilde{\sigma}_{j,i}\tilde{\sigma}_{k,i}g_{k,i}^{-1}\tilde{\sigma}_{k,j}^{-1}\hspace{1em}&k<j<i
		\end{cases}\\
		&=\begin{cases}
			0\hspace{1em}&j<k<i\\
			g_{k,i}\tilde{\sigma}_{k,j}g_{k,i}^{-1}\tilde{\sigma}_{k,j}^{-1}\hspace{1em}&k<j<i
		\end{cases}\\
		&=\begin{cases}
			0\hspace{1em}&j<k<i\\
			g_{k,i}g_{j,i}^{-1}\hspace{1em}&k<j<i
		\end{cases}
	\end{align*}
	Therefore:
	\begin{align*}
		\phi(\tilde{e}_{i,k,j})&=
		\begin{cases}
			0-0\hspace{1em}&j<k<i\\
			g_{k,i}g_{k,j}^{-1}-g_{k,i}g_{j,i}^{-1}\hspace{1em}&k<j<i
		\end{cases}\\
		&=\begin{cases}
			0\hspace{1em}&j<k<i\\
			g_{k,i}-g_{k,j}-(g_{k,i}-g_{j,i})\hspace{1em}&k<j<i\\
		\end{cases}\\
		&=\begin{cases}
			0\hspace{1em}&j<k<i\\
			g_{j,i}-g_{k,j}\hspace{1em}&k<j<i
		\end{cases}
	\end{align*}
\end{proof}
\subsection{Order of $[\phi]$}
To determine the order of $[\phi]$ we build the extensions $G_{n}^{t}$ of $S_{n}$ by $\mathbb{Z}^{n\choose 2}$ which correspond to the class of $t\cdot [\phi]$ and show that if $t\in 2\mathbb{Z}$ then $G_{n}^{t}$ is a split extension. To construct extensions corresponding to multiples of $[\phi]$ consider Theorem \ref{thm: definition of cohom class}. Notice the image of $\phi$ is determined by lifting the relations of $S_{n}$ to relations in $G_{n}$ since we are using the resolution from the universal cover of a $\mathcal{K}(S_{n},1)$-space. 
\paragraph{Remark}
Recall from section \ref{group cohomology}, the extension corresponding to a cohomology class $[\theta]$ is determined by $E_{\theta}=\mathbb{Z}^{n\choose2}\ltimes_{\theta}S_{n}$ where multiplication of group elements is: 
$$(a,g)\cdot (b,h)=(a+g\cdot b+\theta(g,h),gh)$$
To prove $G_{n}^{t}$ is an extension of $S_{n}$ by $\mathbb{Z}_{2}^{n\choose2}$ corresponding to $t\cdot \phi$, it suffices to show $G_{n}^{t}\approx E_{t\cdot\phi}$. The following Lemma proves this isomorphism.
\begin{table}
	\begin{center}
		\caption{Relations of $G_{n}^{t}$}\label{Table: Relations of G^t}
		\vspace{1em}
		\begin{tabular}{lcl}
			
			\vspace{.5em}
			$R^{t}1$:&$[g_{i,j},g_{k,\ell}]=1$& for all $i,j,k,\ell$\\
			
			\vspace{.5em}
			$R^{t}2$:& $\tilde{\sigma}_{i,j}^{2}=g_{i,j}^{t}$ &for all $i$\\
				
			\vspace{1em}
			$R^{t}3$:& 
			$\tilde{\sigma}_{i,j}\tilde{\sigma}_{k,j}\tilde{\sigma}_{i,j}^{-1}=\tilde{\sigma}_{i,k}$
			&if\hspace{1em}$k<i<j;\hspace{.5em}i<j<k;\hspace{.5em}\text{or}\hspace{.5em}j<k<i$\\
				
			\vspace{.5em}
			$R^{t}4$:& $\tilde{\sigma}_{i,j}^{-1}\tilde{\sigma}_{j,k}\tilde{\sigma}_{i,j}=\tilde{\sigma}_{i,k}$ & 
			if\hspace{1em}$i<k<j;\hspace{.5em}j<i<k;\hspace{.5em}\text{or}\hspace{.5em}k<j<i$\\
				
			\vspace{.5em}
			$R^{t}5$:&$[\tilde{\sigma}_{i,j},\tilde{\sigma}_{k,\ell}]=\begin{cases}
				g_{i,k}^{t}g_{i,\ell}^{-t}g_{j,k}^{-t}g_{j,\ell}^{t}\\
				g_{k,i}^{-t}g_{k,j}^{t}g_{i,\ell}^{t}g_{\ell,j}^{-t}\\
				1
			\end{cases}$ & $\begin{aligned}
					i<k<j<\ell\\
					k<i<\ell<j\\
					\text{otherwise}
				\end{aligned}$\\
				
			\vspace{.5em}
			$R^{t}6$:&$\tilde{\sigma}_{i,j}g_{k,\ell}\tilde{\sigma}_{i,j}^{-1}=
			g_{\sigma_{i,j}(k),\sigma_{i,j}(\ell)}$&for all $i,j,k,\ell\in\{1,\ldots,n\}$\\
		\end{tabular}\\
	\end{center}
\end{table}
\begin{lemma}\label{lemma: cohom class t phi}
	Let $G_{n}^{t}$ be the group with generating set $\{\tilde{\sigma}_{i,j}\}\cup\{g_{k,\ell}\}$ for all $i,j,k,\ell$ such that $1\leq i<j\leq n$ and $1\leq k<\ell\leq n$ with relations given in Table \ref{Table: Relations of G^t}. Then $G_{n}^{t}$ is the extension of $S_{n}$ by $\mathbb{Z}^{n\choose2}$ corresponding to the cohomology class of $t\cdot [\phi]\in H_{}^{2}(S_{n};\mathbb{Z}^{n\choose2})$. 
\end{lemma}
\begin{proof}
	We first prove:
	\begin{equation}\label{Diagram: Commuting Diagram for G_{n}^{t}} 
		\begin{tikzcd}
		0\arrow{r}&\mathbb{Z}^{n\choose2}\arrow{r}{\iota_{t}}\arrow{d}{\id_{\mathbb{Z}^{n\choose2}}}&G_{n}^{t}\arrow{r}{\pi_{t}}\arrow{d}{f_{t}}&S_{n}\arrow{r}\arrow{d}{\id_{S_{n}}}&1\\
		0\arrow{r}&\mathbb{Z}^{n\choose2}\arrow{r}{\iota_{t}^{'}}&E_{t\cdot\phi}\arrow{r}{\pi_{t}^{'}}&S_{n}\arrow{r}&1
	\end{tikzcd}
	\end{equation}
	is a commuting diagram of short exact sequences. Since $\mathbb{Z}^{n\choose2}=\langle \{g_{i,j}\}_{1\leq i<j\leq n}\mid [g_{i,j},g_{k,\ell}]=1\rangle$, define a map $\iota_{t}:\mathbb{Z}^{n\choose2}\to G_{n}^{t}$ by $\iota_{t}(g_{i,j})=g_{i,j}$ for all $1\leq i<j\leq n$. Since $g_{i,j}$ is a generator of $G_{n}^{t}$ and relation $R^{t}1$, $\iota_{t}$ is a well defined homomorphism. Recall $\iota^{'}:\mathbb{Z}^{n\choose2}\to E_{t\cdot\phi}$ is defined by $\iota_{t}^{'}(g_{i,j})=(g_{i,j},1_{S_{n}})$; therefore $\iota_{t}^{'}$ is well defined and injective.
	
	Now, define $\pi_{t}:G_{n}^{t}\to S_{n}$ by the following:
	$$\pi_{t}(s)=\begin{cases}
		1\hspace{2em}&\text{if }s\in\{g_{i,j}\}\\
		\sigma_{i,j}\hspace{2em}&\text{if }s\in\{\tilde{\sigma}_{i,j}\}
	\end{cases}$$
	Note that evaluating $\pi_{t}$ on the relations given in Table \ref{Table: Relations of G^t} shows $\pi_{t}$ is well defined. Furthermore, since $\pi_{t}(\tilde{\sigma}_{i,j})=\sigma_{i,j}$, $\pi_{t}$ is surjective. Recall from section \ref{group cohomology}, $\pi_{t}^{'}:E_{t\cdot\phi}\to S_{n}$ defined by $\pi_{t}^{'}((a,g))=g$ is a well defined, surjective homomorphism.
	Define $f_{t}:G_{n}^{t}\to E_{t\cdot\phi}$ on the generators of our presentation by:
	$$f_{t}(s)=\begin{cases}
		(s,1_{S_{n}})\hspace{1em}&\text{if }s\in\{g_{i,j}\}_{1\leq i<j\leq n}\\
		(0,s)\hspace{1em}&\text{if }s\in\{\tilde{\sigma}_{i,j}\}_{1\leq i<j\leq n}
	\end{cases}$$
	By the definitions of $\iota_{t}$, $\iota_{t}^{'}$, and $f_{t}$:
	\begin{align*}
		f_{t}\circ\iota_{t}(g_{i,j})&=f(\iota_{t}(g_{i,j}))\\
		&=f_{t}(g_{i,j})\\
		&=(g_{i,j},1_{S_{n}})\\
		&=\iota_{t}^{'}(g_{i,j})
	\end{align*}
	Therefore $f_{t}\circ\iota_{t}=\iota_{t}^{'}\circ\id_{\mathbb{Z}^{n\choose2}}$. Suppose $a\in\ker\iota_{t}$, then $f_{t}\circ\iota_{t}(a)=\iota_{t}^{'}(a)$ and $f_{t}\circ\iota_{t}(a)=(0,1_{S_{n}})$. Since $\iota_{t}^{'}$ is injective, $a=0$ and $\iota_{t}$ is injective. 
	Let $s$ be a generator of $G_{n}^{t}$. By the definitions of $f_{t}$, $\phi_{t}$, and $\pi_{t}^{'}$, if $s\in\{g_{i,j}\}_{1\leq i<j\leq n}$:
	\begin{align*}
		\pi_{t}^{'}\circ f_{t}(g_{i,j})&=\pi_{t}^{'}(f_{t}(g_{i,j}))\\
		&=\pi_{t}^{'}((g_{i,j},1_{S_{n}}))\\
		&=1_{S_{n}}\\
		&=\pi_{t}(g_{i,j})
	\end{align*}
	Similarly, if $s\in\{\tilde{\sigma}_{i,j}\}_{1\leq i<j\leq n}$:
	\begin{align*}
		\pi_{t}^{'}\circ f_{t}(\tilde{\sigma}_{i,j})&=\pi_{t}^{'}(f_{t}(\tilde{\sigma}_{i,j}))\\
		&=\pi_{t}^{'}((0,\sigma_{i,j}))\\
		&=\sigma_{i,j}\\
		&=\pi_{t}(\tilde{\sigma}_{i,j})
	\end{align*}
	Therefore $f_{t}$ makes the diagram \eqref{Diagram: Commuting Diagram for G_{n}^{t}} commute. It remains to show the top row of \eqref{Diagram: Commuting Diagram for G_{n}^{t}} is exact. By the definitions of $\iota_{t}$ and $\pi_{t}$, $\im\iota_{t}\subseteq\ker\pi_{t}$. Furthermore, by adding the relation $g_{i,j}=1$ to the relations in Table \ref{Table: Relations of G^t} we obtain a presentation for $S_{n}$. Therefore $\ker\pi_{t}$ is the normal closure of $\im\iota_{t}$. Note the image of $\iota_{t}$ is generated by the $g_{i,j}$'s in $G_{n}^{t}$. By relation $R^{t}6$, the $g_{i,j}$'s generate a normal subgroup of $G_{n}^{t}$. Thus $\im\iota_{t}$ is normal in $G_{n}^{t}$. In particular, the normal closure of $\im\iota_{t}$ is equal to $\im\iota_{t}$. Therefore $\im\iota_{t}=\ker\pi_{t}$ and the top row of \eqref{Diagram: Commuting Diagram for G_{n}^{t}} is exact. Thus $G_{n}^{t}$ is an extension of $S_{n}$ by $\mathbb{Z}^{n\choose2}$.
	
	Since $E_{t\cdot\phi}$ is an extension of $S_{n}$ by $\mathbb{Z}^{n\choose2}$, the bottom row of \eqref{Diagram: Commuting Diagram for G_{n}^{t}} is exact as well. By the 5-Lemma, $G_{n}^{t}\approx E_{t\cdot\phi}$. Furthermore, since $f_{t}$ makes the diagram commute, $G_{n}^{t}$ and $E_{t\cdot\phi}$ are equivalent group extensions. Therefore $G_{n}^{t}$ corresponds to the cohomology class of $t\cdot[\phi]\in H^{2}(S_{n};\mathbb{Z}^{n\choose2})$.
\end{proof}

\paragraph{Remark}
An alternative, and more formal, process of constructing the cohomology class corresponding to $t\cdot[\phi]$ in $H^{2}(S_{n};\mathbb{Z}^{n\choose2})$ would be to construct new extensions from $G_{n}$ using a technique analogous to Baer sums. Note that Baer sums do not generalize when the cokernel of an extension is non-abelian, however no obstructions arise when taking the pullback construction of equivalent extensions with the same group presentation.
	\begin{theorem}\label{splitting of G^t}
		If $t\in2\mathbb{Z}$, then the group extension:
		$$0\to\mathbb{Z}^{n\choose2}\to G_{n}^{t}\to S_{n}\to 1$$
		is split.
	\end{theorem}
	\begin{proof}
		First, note that the existence of a splitting is independent of group presentation for $S_{n}$. Therefore it suffices to define a homomorphism $\omega:S_{n}\to G_{n}^{t}$ using the presentation \eqref{symmetric group generated by adjacent transpositions} of $S_{n}$. Set: $$\omega(\sigma_{i})=\tilde{\sigma}_{i,i+1}g_{i,i+1}^{-t/2}$$
		First note that:
		\begin{align*}
			\omega(\sigma_{i})^{2}&=\tilde{\sigma}_{i,i+1}g_{i,i+1}^{-t/2}\tilde{\sigma}_{i,i+1}g_{i,i+1}^{-t/2}\\
			&=\tilde{\sigma}_{i,i+1}g_{i,i+1}\tilde{\sigma}_{i,i+1}^{-1}\tilde{\sigma}_{i,i+1}^{2}g_{i,i+1}\\
			&=g_{i,i+1}^{-t/2}g_{i,i+1}^{t}g_{i,i+1}^{-t/2}\\
			&=1
		\end{align*}
		Now, if $|i-j|>1$, then by the relations in Table \ref{Table: Relations of G^t}:
		$$\omega(\sigma_{i})\omega(\sigma_{j})\omega(\sigma_{i})^{-1}\omega(\sigma_{j})^{-1}=1$$
		Thus it only remains to show the braid relation is preserved under $\omega$. Applying relations $R^{t}1$, $R^{t}2$, and $R^{t}6$ together with $R^{t}3$ on the image we get:
		\begin{align*}
			&\tilde{\sigma}_{i,i+1}^{}g_{i,i+1}^{-t/2}\tilde{\sigma}_{i+1,i+2}^{}g_{i+1,i+2}^{-t/2}\tilde{\sigma}_{i,i+1}^{}g_{i,i+1}^{-t/2}g_{i+1,i+2}^{t/2}\tilde{\sigma}_{i+1,i+2}^{-1}g_{i,i+1}^{t/2}\tilde{\sigma}_{i,i+1}^{-1}g_{i+1,i+2}^{t/2}\tilde{\sigma}_{i+1,i+2}^{-1}\\
			&=g_{i,i+1}^{-t/2}\tilde{\sigma}_{i,i+1}^{}\tilde{\sigma}_{i+1,i+2}^{}\tilde{\sigma}_{i,i+1}^{-1}\tilde{\sigma}_{i,i+1}^{}g_{i+1,i+2}^{-t/2}\tilde{\sigma}_{i,i+1}^{-1}\tilde{\sigma}_{i,i+1}^{2}g_{i,i+1}^{-t/2}g_{i+1,i+2}^{t/2}\tilde{\sigma}_{i+1,i+2}^{-1}g_{i,i+1}^{t/2}\tilde{\sigma}_{i,i+1}^{-1}\tilde{\sigma}_{i+1,i+2}^{-1}g_{i+1,i+2}^{-1}\\
			&=g_{i,i+1}^{-t/2}\tilde{\sigma}_{i,i+2}^{}g_{i,i+2}^{-t/2}g_{i,i+1}^{t}g_{i,i+1}^{-t/2}g_{i+1,i+2}^{t/2}\tilde{\sigma}_{i+1,i+2}^{-2}\tilde{\sigma}_{i+1,i+2}^{}g_{i,i+1}^{t/2}\tilde{\sigma}_{i+1,i+2}^{-1}\tilde{\sigma}_{i+1,i+2}^{}\tilde{\sigma}_{i,i+1}^{-1}\tilde{\sigma}_{i+1,i+2}^{-1}g_{i+1,i+2}^{t/2}\\
			&=g_{i,i+1}^{-t/2}\tilde{\sigma}_{i,i+2}^{}g_{i,i+2}^{-t/2}g_{i,i+1}^{t/2}g_{i+1,i+2}^{t/2}g_{i+1,i+2}^{-t}g_{i,i+2}^{t/2}\tilde{\sigma}_{i+1,i+2}^{}\tilde{\sigma}_{i,i+1}^{-1}\tilde{\sigma}_{i+1,i+2}^{-1}g_{i+1,i+2}^{t/2}\\
			&=g_{i,i+1}^{-t/2}\tilde{\sigma}_{i,i+2}g_{i,i+1}^{t/2}g_{i+1,i+2}^{-t/2}\tilde{\sigma}_{i+1,i+2}^{2}\tilde{\sigma}_{i+1,i+2}^{-1}\tilde{\sigma}_{i,i+1}^{-1}\tilde{\sigma}_{i+1,i+2}^{}\tilde{\sigma}_{i+1,i+2}^{-2}g_{i+1,i+2}^{t/2}\\
			&=g_{i,i+1}^{-t/2}\tilde{\sigma}_{i,i+2}^{}g_{i,i+1}^{t/2}g_{i+1,i+2}^{-t/2}g_{i+1,i+2}^{t}\tilde{\sigma}_{i,i+2}^{-1}g_{i+1,i+2}^{-t}g_{i+1,i+2}^{-t/2}\\
			&=g_{i,i+1}^{-t/2}\tilde{\sigma}_{i,i+2}^{}g_{i,i+1}^{t/2}\tilde{\sigma}_{i,i+2}^{-1}\tilde{\sigma}_{i,i+2}^{}g_{i+1,i+2}^{t/2}\tilde{\sigma}_{i,i+2}^{-1}g_{i+1,i+2}^{-t/2}\\
			&=g_{i,i+1}^{-t/2}g_{i+1,i+2}^{t/2}g_{i,i+1}^{t/2}g_{i+1,i+2}^{-t/2}\\
			&=1
		\end{align*}
		Thus $\omega$ is well defined and the short exact sequence is split.
	\end{proof}

\section{$\Z_{n}$ and $B_{n}[4]$}
\subsection{Cohomology class corresponding to $\mathcal{Z}_{n}$}
Recall the group extension
	$$\begin{tikzcd}
		0\arrow{r}&\mathbb{Z}_{2}^{n\choose2}\arrow{r}{\iota}&\mathcal{Z}_{n}\arrow{r}{\pi}&S_{n}\arrow{r}&1
	\end{tikzcd}$$
Since $\mathbb{Z}_{2}^{n\choose2}\approx\mathcal{PZ}_{n}$ and $\mathcal{Z}_{n}\approx B_{n}/B_{n}[4]$, the action of $S_{n}$ on $\mathbb{Z}_{2}^{n\choose2}$ is induced by the conjugation action of $B_{n}$ on $PB_{n}$. Furthermore, $\mathbb{Z}_{2}^{n\choose2}\approx \mathcal{PZ}_{n}$, so each $\bar{g}_{ij}$ corresponds to the pure braid between strands $i$ and $j$. 
\begin{table}
	\begin{center}
		\caption{Relations of $\mathcal{Z}_{n}$}\label{Relations of Z}
		\vspace{1em}
		\begin{tabular}{lcl}
			
			\vspace{.5em}
			$\mathfrak{R}0$:&$\bar{g}_{i,j}^{2}=1$& for all $i,j$\\
			\vspace{.5em}
			$\mathfrak{R}1$:&$[\bar{g}_{i,j},\bar{g}_{k,\ell}]=1$& for all $i,j,k,\ell$\\
			\vspace{.5em}
			$\mathfrak{R}2$:& $\bar{\sigma}_{i,j}^{2}=\bar{g}_{i,j}$ &for all $i,j$\\
				
			\vspace{1em}
			$\mathfrak{R}3$:& 
			$\bar{\sigma}_{i,j}\bar{\sigma}_{k,j}\bar{\sigma}_{i,j}^{-1}=\bar{\sigma}_{i,k}$
			&if\hspace{1em}$k<i<j;\hspace{.5em}i<j<k;\hspace{.5em}\text{or}\hspace{.5em}j<k<i$\\
				
			\vspace{.5em}
			$\mathfrak{R}4$& $\bar{\sigma}_{i,j}^{-1}\bar{\sigma}_{j,k}\bar{\sigma}_{i,j}=\bar{\sigma}_{i,k}$ & 
			if\hspace{1em}$i<k<j;\hspace{.5em}j<i<k;\hspace{.5em}\text{or}\hspace{.5em}k<j<i$\\
				
			\vspace{.5em}
			$\mathfrak{R}5$&$[\bar{\sigma}_{i,j},\bar{\sigma}_{k,\ell}]=\begin{cases}
				\bar{g}_{i,k}\bar{g}_{i,\ell}\bar{g}_{j,k}\bar{g}_{j,\ell}\\
				1
			\end{cases}$ & $\begin{aligned}
				&i<k<j<\ell\text{ or }
				k<i<\ell<j\\
				&\text{otherwise}
			\end{aligned}$\\
				
			\vspace{.5em}
			$\mathfrak{R}6$:& $\bar{\sigma}_{i,j}g_{k,\ell}\bar{\sigma}_{i,j}^{-1}=
			g_{\sigma_{i,j}(k),\sigma_{i,j}(\ell)}$&for all $i,j,k,\ell\in\{1,\ldots,n\}$\\
		\end{tabular}\\
	\end{center}
\end{table}
\begin{theorem}\label{thm: presentation of Z}
	Let $A_{n}$ be the subgroup of $G_{n}$ normally generated by $\{g_{i,j}^{2}\}_{1\leq i<j\leq n}$. Then $G_{n}/A_{n}\approx \mathcal{Z}_{n}$.
\end{theorem}
\begin{proof}
	Since $A_{n}$ is generated by $g_{i,j}^{2}$ for $1\leq i<j\leq n$ and $g_{i,j}$ corresponds to the pure braid between the $i^{\text{th}}$ and $j^{\text{th}}$ strands, $A_{n}$ is the subgroup of $G_{n}$ generated by squares of standard generators of pure braids. Recall $B_{n}[4]$ is the subgroup of $PB_{n}$ generated by squares of all elements \cite{B/M}. Since the $g_{i,j}$'s commute in $G_{n}$, the image of $PB_{n}$ is an abelian subgroup of $G_{n}$. Therefore the image of $B_{n}[4]$ in $G_{n}$ is generated by squares of generators of the pure braid group. Therefore $A_{n}$ is the image of $B_{n}[4]$ in $G_{n}$. Thus $G_{n}/A_{n}\approx\Z_{n}$.
\end{proof}
	
\paragraph{Presentation of $\mathcal{Z}_{n}$}
Let $s^{'}$ be the normalized section lifting $\sigma_{i,j}$ to $\bar{\sigma}_{i,j}$ where $\bar{\sigma}$ represents the image of $b_{i,j}$ in $\mathcal{Z}_{n}$. Since $A_{n}$ is the image of $B_{n}[4]$ in $G_{n}$, we obtain a presentation for $\Z_{n}$ by adding the relation $g_{i,j}^{2}=1$ to the relations in Table \ref{Table: Relations of G}. We simplify the notation by taking all winding numbers$\mod2$. Furthermore, we replace $g_{i,j}$ and $\sigma_{i,j}$ with $\bar{g}_{i,j}$ and $\bar{\sigma}_{i,j}$ for clarity. Therefore we get the following generating set for $\Z_{n}$:
	$$\{\bar{g}_{i,j}\}\cup\{\bar{\sigma}_{i,j}\}$$
where $1\leq i<j\leq n$. A full list of relations is given in Table \ref{Relations of Z}. The following theorem completes the proofs of both Theorem \ref{main theorem} and Theorem \ref{def of cocycle}.
\paragraph{Remark} Notice that $\mathfrak{R}4$ can be determined from $\mathfrak{R}3$.
\begin{theorem}\label{thm: cohom class of Z}
	Let $\kappa\in\hom(R_{2},\mathbb{Z}_{2}^{n\choose2})$ be the represenative for the cohomology class in $H^{2}(S_{n};\mathbb{Z}_{2}^{n\choose2})$ corresponding to $\Z_{n}$ as an extension of $S_{n}$ by $\mathbb{Z}_{2}^{n\choose2}$ given by the usual construction. Let $\eta:\mathbb{Z}^{n\choose2}\to\mathbb{Z}_{2}^{n\choose2}$ be the$\mod2$ reduction map. Then $\kappa=\eta\circ\phi$ where $\phi$ is the representative for $G_{n}$ defined in Theorem \ref{thm: cohom of G_{n}}.
\end{theorem}
\begin{proof}
	Define $f:G_{n}\to\mathcal{Z}_{n}$ by $f(\tilde{\sigma}_{i,j})=\bar{\sigma}_{i,j}$ and $f(g_{i,j})=\bar{g}_{i,j}$ for all $i,j$ such that $1\leq i<j\leq n$. Then $f$ is a well defined group homomorphism which commutes with the following diagram:
	$$\begin{tikzcd}
		0\arrow{r}&\mathbb{Z}^{n\choose2}\arrow{r}{\iota_{1}}\arrow{d}{\eta}&G_{n}\arrow{r}{\pi_{1}}\arrow{d}{f}&S_{n}\arrow{r}\arrow{d}{\text{id}}&1\\
		0\arrow{r}&\mathbb{Z}_{2}^{n\choose2}\arrow{r}{\iota}&\mathcal{Z}_{n}\arrow{r}{\pi}&S_{n}\arrow{r}&1
	\end{tikzcd}$$
	Furthermore, by Theorem \ref{thm: definition of cohom class}
	\begin{align*}
		\kappa(\tilde{c}_{i,j})&=s^{'}(\sigma_{i,j})s^{'}(\sigma_{i,j})\\
		\kappa(\tilde{d}_{i,j,k,\ell})=&=s^{'}(\sigma_{i,j})s^{'}(\sigma_{k,\ell})s^{'}(\sigma_{i,j}\sigma_{k,\ell})^{-1}-s^{'}(\sigma_{k,\ell})s^{'}(\sigma_{i,j})s^{'}(\sigma_{k,\ell}\sigma_{i,j})^{-1}\\
		\kappa'(\tilde{e}_{i,k,j})&=s^{'}(\sigma_{i,j})s^{'}(\sigma_{j,k})s^{'}(\sigma_{i,j}\sigma_{j,k})^{-1}-s^{'}(\sigma_{i,k})s^{'}(\sigma_{i,j})s^{'}(\sigma_{i,k}\sigma_{i,j})^{-1}
	\end{align*}
	Since $f$ commutes with the diagram and by the choices of $s$ and $s^{'}$, we have $s^{'}(\sigma_{i,j})=f\circ s(\sigma_{i,j})$. Thus $\kappa=\phi\circ f$. Furthermore, the induced map on cohomology is $\eta$ by the definition of $f$. 
\end{proof}
\subsection{Generating sets of $B_{n}[4]$}
In this section we begin with a well known consequence of Schreier's formula which guarantees the existence of a finite generating set for $B_{n}[4]$. Then we prove Theorem \ref{normal gen set for B_{n}} and give a summary describing the difficulties in refining our normal generating set to a finite generating set.
\begin{prop}\label{schreier's formula}
	If $G$ is a group with a generating set of size $j$ and $H$ is a subgroup of finite index $k$, then there exists a finite generating set for $H$ of size at most $k(j-1)+1$.
\end{prop}
This proposition is a well know fact and a proof can be found in Lyndon and Schupp (page 164 of \cite{L/S}). Since $\Z_{n}$ is an extension of $S_{n}$ by $\mathbb{Z}_{2}^{n\choose2}$, $\Z_{n}$ is finite of order $n!\cdot 2^{n\choose2}$. Furthermore, $B_{n}$ is finitely generated with a generating set of size $n-1$, therefore there exists a generating set for $B_{n}[4]$ of size at most:
	$$n!\cdot 2^{n\choose2}\cdot\left(n-2\right)+1$$
Now we give a proof of Theorem \ref{normal gen set for B_{n}}.
\begin{proof}[Proof of Theorem \ref{normal gen set for B_{n}}]
	Consider the quotient of the Artin presentation of $B_{n}$ given in \eqref{Artin braid} by the normal subgroup generated by $b_{i}^{4}$, $[b_{i}^2,b_{i+1}^2]$, and $[b_{i,i+2}^{2},b_{i+1,i+3}^{2}]$. Then this quotient has a group presentation:
	\begin{equation}\label{Z_n presentation generated by b_i}
		\left\langle b_{1},\ldots, b_{n-1}\middle\mid\begin{aligned} \hspace{.5em}&b_{i}^{4}=1,\hspace{.5em}[b_{i}^{2},b_{i+1}^{2}]=1,\hspace{.5em}[b_{i,i+2}^{2},b_{i+1,i+3}^{2}]=1
			\\
		&\hspace{.5em}b_{i}b_{i+1}b_{i}=b_{i+1}b_{i}b_{i+1},\hspace{.5em}
		b_{i}b_{j}=b_{j}b_{i}\text{ if }|i-j|>1
		\end{aligned}\right\rangle
	\end{equation}
	Now, consider the presentation for $\Z_{n}$ with generators $\{\bar{g}_{i,j}\}$ and $\{\bar{\sigma}_{i,j}\}$ for all $1\leq i<j\leq n$ with relations given by Table \ref{Relations of Z}. By relation $\mathfrak{R}2$, we can replace $\bar{g}_{i,j}$ with $\bar{\sigma}_{i,j}^{2}$ for all $1\leq i<j\leq n$. Therefore $\Z_{n}$ is generated by $\{\bar{\sigma}_{i,j}\}$ where $1\leq i<j\leq n$. Furthermore, by the case of relation $\mathfrak{R}3$ with $i<j<k$, $\{\bar{\sigma}_{i,j}\}$ is generated by the set $\{\bar{\sigma}_{i,i+1}\}$. Thus we can take a generating set of $\{\bar{\sigma}_{i,i+1}\}_{1\leq i\leq n-1}$ for $\Z_{n}$. 
	
	Now, since $\bar{\sigma}_{i,j}^{2}=\bar{g}_{i,j}$ and $\bar{g}_{i,j}^{2}=1$, therefore we can replace $\mathfrak{R}0$ and $\mathfrak{R}2$ with $\bar{\sigma}_{i,i+1}^{4}=1$ in $\Z_{n}$. By relation $\mathfrak{R}3$ and $\mathfrak{R}4$, we have $\bar{\sigma}_{i,i+1}\bar{\sigma}_{i+1,i+2}\bar{\sigma}_{i,i+1}^{-1}=\bar{\sigma}_{i,i+2}$ and $\bar{\sigma}_{i+1,i+2}^{-1}\bar{\sigma}_{i,i+1}\bar{\sigma}_{i+1,i+2}=\bar{\sigma}_{i,i+2}$ respectively. Therefore:
	$$\bar{\sigma}_{i,i+1}\bar{\sigma}_{i+1,i+2}\bar{\sigma}_{i,i+1}^{-1}=\bar{\sigma}_{i+1,i+2}^{-1}\bar{\sigma}_{i,i+1}\bar{\sigma}_{i+1,i+2}$$
	and we can replace relations $\mathfrak{R}3$ and $\mathfrak{R}4$ with the relation braid relation. Furthermore, $\mathfrak{R}5$ implies $[\bar{\sigma}_{i,i+1},\bar{\sigma}_{j,j+1}]=1$ for $|i-j|>1$. Replacing $\bar{g}_{i,j}$ with $\bar{\sigma}_{i,j}^{2}$ in relation $\mathfrak{R}1$ and restricting to $\{\bar{\sigma}_{i,i+1}\}_{1\leq i\leq n-1}$ we get $[\bar{\sigma}_{i,i+1}^{2},\bar{\sigma}_{i+1,i+2}^{2}]=1$ and $[\bar{\sigma}_{i,i+2}^{2},\bar{\sigma}_{i+1,i+3}^{2}]=1$. Thus we get the same presentation as \eqref{Z_n presentation generated by b_i} with $\bar{\sigma}_{i,i+1}$ replacing $b_{i}$. Recall, $\Z_{n}=B_{n}/B_{n}[4]$; therefore the relations of \eqref{Z_n presentation generated by b_i} which are not relations of \eqref{Artin braid} normally generate $B_{n}[4]$. 
\end{proof}
\paragraph{Finite generating set for $B_{n}[4]$}
Proposition \ref{schreier's formula} guarantees the existence of a finite generating set for $B_{n}[4]$, however the size of this generating set grows faster than exponentially in $n$. We would hope for a generating set which has polynomial growth in $n$ and the natural method of finding this would be to add necessary generators to our normal generating set until we obtain a finite generating set. Notice that our normal generating set contains only squares and commutators of pure braids. Therefore we would hope to take the union of a finite generating sets for squares of pure braids and the commutator subgroup of the pure braid group.

Cohen, Falk, and Randell proved there exists an epimorphism from the pure braid group to the free group of rank 2 \cite{C/F/R}. Since the commutator subgroup of the free group of rank 2 is not finitely generated, the existence of an epimorphism implies the commutator subgroup of the pure braid group is not finitely generated. Therefore we must first determine the intersection of the commutator subgroup of the pure braid group with the subgroup generated by squares of pure braids.

For odd integers $m\geq 1$ and $k\in\{2,4\}$, the work of Apple, Bloomquist, Gravel, and Holden determine certain quotient structures of $B_{n}[m]/B_{n}[km]$ \cite{A/B/G/H}. An alternative approach to determining a finite generating set for $B_{n}[4]$ would be to determine $B_{n}[4]$ as a group extensions and use the methods similar to section \ref{section on building presentations} to construct a group presentation. The first step in this method is to understand $B_{n}[8]$ or $B_{n}[16]$. 


\newpage
\begin{thebibliography}{11}
	\bibitem{A/B/G/H}
	Jessica Appel, Wade Bloomquist, Katie Gravel, Annie Holden. On quotients of congruence subrgroups of braid groups. https://arxiv.org/abs/2011.13876, 2020
	
	\bibitem{Birman}
	Joan S. Birman. \textit{Braids, links, and mapping class groups.} Princeton University Press, Princeton, N.J.; University of Tokyo Press, Tokyo, 1974. Annals of Mathematica Studies, No. 82.
		
	\bibitem{BKL}
	Joan Birman, Ki Hyoung Ko, and Sang Jin Lee. A new approach to the word and conjugacy problems in the braid groups. \textit{Adv. Math}, 139(2):322-353, 1998
		
	\bibitem{B/M}
	Tara E. Brendle and Dan Margalit. The level 4 braid group. \text{J. Reine Angew. Math.}, 735: 249-264, 2018.
		
	\bibitem{Brown}
	Kenneth S. Brown, \textit{Cohomology of Groups}, volume 87 of \textit{Graduate Texts in Mathematics}. Springer-Verlag, New York-Berlin, 1982.
		
		
	\bibitem{C/F/R}
	Daniel C. Cohen, Michael Falk, and Richard Randell. “Pure Braid Groups Are Not Residually Free.” Configuration Spaces (2012): 213–230.
		
	\bibitem{E/M}
	Samuel Eilenberg and Saunders Maclane. Cohomology Theory in Abstract Groups. II: Group Extensions with a Non-Abelian Kernel. Annals of Mathematics, vol. 48, no. 2, 1947, pp. 326–341. 
		
	\bibitem{F/M}
	Benson Farb and Dan Margalit,
	\textit{A primer on mapping class groups},
	volume 49 of \textit{Princeton Mathematical Series}. Princeton University Press, Princeton NJ, 2012
		
	\bibitem{K/M}
	Kevin Kordek and Dan Margalit. Representation stability in the level 4 braid group. https://arxiv.org/abs/1903.03119, 2019
		
	\bibitem{L/S}
	Roger C. Lyndon and Paul E. Schupp, \textit{Combinatorial Group Theory}, volume 89 of \textit{Ergebnisse der Mathematik und ihrer Grenzgebiete}. Springer-Verlag, Berlin-Heidelberg-New York 1977.
		
	\bibitem{Styl}
	Charalampos Stylianakis. Congruence Subgroups of braid groups.\textit{International Journal of Algebra and Computation}, vol 28, no 2. 2018, pp. 345-365
\end{thebibliography}
\end{document}